\documentclass[11pt]{article}

\usepackage[T1]{fontenc}
\usepackage[letterpaper,margin=1in]{geometry}
\usepackage{amsmath,amssymb,amsthm,mathtools}
\usepackage{graphicx}
\usepackage{float}
\usepackage{placeins}
\usepackage{tikz}
\usetikzlibrary{arrows.meta,positioning,calc}
\usepackage{algorithm}
\usepackage{algpseudocode}
\usepackage{array}
\usepackage{booktabs}
\usepackage{enumitem}
\usetikzlibrary{arrows.meta,positioning,calc,fit,backgrounds}
\usepackage{microtype}
\usepackage{xurl}
\usepackage[hidelinks]{hyperref}
\expandafter\def\expandafter\UrlBreaks\expandafter{\UrlBreaks
  \do\-\do\_\do\.\do\/\do\=\do\&}

\hypersetup{
  pdftitle={A Counterexample to the Convergence of Three-Block ADMM with an Identity Third Constraint Block},
  pdfauthor={Kenan Xu, and Xiangfeng Wang},
  pdfsubject={AI-assisted discovery and exact certification of periodic nonconvergence for identity-slack three-block ADMM},
  pdfkeywords={ADMM, multi-block convex optimization, slack-variable reformulation, KKT conditions, periodic nonconvergence, piecewise-affine iteration, multiplier relaxation, research harness, AI-assisted discovery, AI for mathematics}
}

\theoremstyle{plain}
\newtheorem{theorem}{Theorem}[section]
\newtheorem{proposition}[theorem]{Proposition}

\newtheorem{corollary}[theorem]{Corollary}

\theoremstyle{remark}

\numberwithin{equation}{section}
\setlist[itemize]{leftmargin=2em,itemsep=0.35ex,topsep=0.6ex}
\setlist[enumerate]{leftmargin=2.2em,itemsep=0.4ex,topsep=0.6ex}

\newcommand{\R}{\mathbb R}
\newcommand{\Q}{\mathbb Q}
\newcommand{\diag}{\operatorname{diag}}
\newcommand{\ind}{\iota}
\newcommand{\argmin}{\operatorname*{arg\,min}}

\title{AI-Assisted Discovery and Construction of a Counterexample to the Convergence of Three-Block ADMM with the Identity Matrix as Its Third Constraint Block\thanks{Codes, certificates, prompts, and workspace manifests are available at \url{https://github.com/ConanXu-math/identity-slack-admm-cycle-certificate}.}\ \thanks{We would like to express our sincere gratitude to Professor Bingsheng He for encouraging me to explore the problem addressed in this paper using AI, and for the in-depth discussions we had throughout the process, from which we have benefited immensely.}}

\author{
Kenan Xu\thanks{School of Mathematical Sciences, East China Normal University, Shanghai 200241, P.R. China.} \and 
Xiangfeng Wang\thanks{Key Laboratory of Mathematics and Engineering Applications (MoE) and School of Mathematical Sciences, East China Normal University, Shanghai 200241, P.R. China.} 
}
\date{}

\begin{document}

\maketitle

\begin{abstract}
  The alternating direction method of multipliers (ADMM), as a landmark algorithm, has attracted tremendous research attention and extensive practical applications over the past two decades \cite{BPCPE2011}.
  It is well known that, although the two-block ADMM enjoys well-established theoretical convergence guarantees, its direct extension to the three-block case may fail to converge, as demonstrated by existing counterexamples \cite{CHYY2016}.
  However, to the best of our knowledge, the case in which the third constraint
  block is the identity remains unresolved: the existing literature gives neither
  a general convergence proof nor a counterexample for this subclass.
  In this paper, we give a negative answer: direct three-block ADMM may fail even when the first two blocks are strongly convex quadratics.
  Using Codex with GPT-5.6 Sol, we  construct an explicit rational
  counterexample candidate and verify it along a piecewise-affine reduction
  path; exact checks show that direct three-block ADMM on this instance
  produces a bounded nonconvergent orbit of period \(66\).
  Within the same Codex workflow, we further guide a study of multiplier
  relaxation and clarify when convergence can be restored at the fixed-instance
  and class levels: a problem-dependent small dual step can restore
  convergence, whereas no positive relative step works uniformly over the
  whole class.
  Furthermore, we also test the recent Kimi Code with Kimi K3 model without the Codex candidate or project-specific route guidance; along a different path it produces an
  exact locally attracting period-\(23\) certificate, convertible to an
  equivalent all-identity instance.
  The comparison suggests that different research-harness configurations can shape the mathematical objects explored and the certificates pursued.
\end{abstract}

\medskip

\section{The Open Problem}
\label{sec:introduction}

\subsection{The Identity-Slack Question}
\label{sec:identity-slack}

Consider the linearly constrained three-block convex program
\begin{equation}
\label{eq:general-three-block}
  \begin{aligned}
    \min_{x_1,x_2,x_3}\quad
    &\theta_1(x_1)+\theta_2(x_2)+\theta_3(x_3)\\
    \text{s.t.}\quad
    &A_1x_1+A_2x_2+A_3x_3=b,
  \end{aligned}
\end{equation}
where, for \(i=1,2,3\),
\(x_i\in\R^{n_i}\), \(A_i\in\R^{m\times n_i}\), \(b\in\R^m\), and
\(\theta_i:\R^{n_i}\to(-\infty,+\infty]\) is proper, closed, and convex.
Block constraints may be absorbed into \(\theta_i\) through indicator
functions, and the solution set is assumed nonempty.
Although two-block ADMM
is convergent under standard assumptions
\cite{BPCPE2011,EB1992}, direct three-block ADMM can fail for
general coefficient matrices.
Both the zero-objective feasibility example
and the strongly convex example in \cite{CHYY2016} lack an identity
third block.
We study the distinct subclass \(n_3=m\), \(A_3=I_m\).
Typically, this can be considered as the closest case compared with classical two-block problem.

This structure usually arises from the inequality-constrained problem
\begin{equation}
\label{eq:inequality-problem}
  \begin{aligned}
    \min_{x,y}\quad &F(x)+G(y)\\
    \text{s.t.}\quad &Ax+By\leq b,
  \end{aligned}
\end{equation}
where \(F:\R^{n_x}\to(-\infty,+\infty]\) and
\(G:\R^{n_y}\to(-\infty,+\infty]\) are proper, closed, and convex,
\(A\in\R^{m\times n_x}\), \(B\in\R^{m\times n_y}\), and \(b\in\R^m\).
All vector inequalities are understood componentwise, and
\(\R_+^m:=\{u\in\R^m:u\geq0\}\).
For a nonempty closed convex set \(C\), let \(\ind_C\) denote its indicator:
\(\ind_C(u)=0\) for \(u\in C\) and \(\ind_C(u)=+\infty\) otherwise.
Introducing a nonnegative slack variable gives
\begin{equation}
\label{eq:slack-problem}
  \begin{aligned}
    \min_{x,y,z}\quad &F(x)+G(y)+\ind_{\R_+^m}(z)\\
    \text{s.t.}\quad &Ax+By+z=b.
  \end{aligned}
\end{equation}
This is \eqref{eq:general-three-block} with
\((x_1,x_2,x_3)=(x,y,z)\) and \((A_1,A_2,A_3)=(A,B,I_m)\); the last
subproblem is projection onto \(\R_+^m\).  We consider the direct slack-last
sweep, rather than a grouped two-block reformulation \cite{LGSP2026}.
For \(\beta>0\) and \(\lambda\in\R^m\), the augmented Lagrangian function is
defined as
\begin{equation}
\label{eq:lagrangian}
  \mathcal L_\beta(x,y,z,\lambda)
  =F(x)+G(y)+\ind_{\R_+^m}(z)
  -\lambda^\top r+\frac{\beta}{2}\lVert r\rVert^2,
  \qquad r=Ax+By+z-b\in\R^m.
\end{equation}
Assuming that the subproblems attain solutions, direct three-block ADMM in the
order \(x\to y\to z\to\lambda\) is
\begin{subequations}
\label{eq:direct-admm}
\begin{align}
  x^{k+1}&\in\argmin_x
    \mathcal L_\beta(x,y^k,z^k,\lambda^k),\label{eq:direct-x}\\
  y^{k+1}&\in\argmin_y
    \mathcal L_\beta(x^{k+1},y,z^k,\lambda^k),\label{eq:direct-y}\\
  z^{k+1}&\in\argmin_z
    \mathcal L_\beta(x^{k+1},y^{k+1},z,\lambda^k),\label{eq:direct-z}\\
  \lambda^{k+1}&=\lambda^k-\beta
    (Ax^{k+1}+By^{k+1}+z^{k+1}-b).\label{eq:direct-lambda}
\end{align}
\end{subequations}
we can employ the normal-cone convention
\[
  N_C(u)=\{v:\langle v,w-u\rangle\leq0\ \text{for every }w\in C\},
  \qquad u\in C.
\]
With the sign convention in \eqref{eq:lagrangian}, a KKT point of
\eqref{eq:slack-problem} satisfies
\begin{equation}
\label{eq:kkt}
  \begin{gathered}
    A^\top\lambda^\star\in\partial F(x^\star),\qquad
    B^\top\lambda^\star\in\partial G(y^\star),\\
    Ax^\star+By^\star+z^\star=b,\qquad
    z^\star\in\R_+^m,\qquad
    \lambda^\star\in N_{\R_+^m}(z^\star).
  \end{gathered}
\end{equation}
As mentioned above, it is well known that, although the two-block ADMM enjoys well-established theoretical convergence guarantees, its direct extension to the three-block case may fail to converge, as demonstrated by existing counterexamples \cite{CHYY2016}.
However, available multi-block convergence results impose additional regularity or
parameter assumptions \cite{CHY2017,LMZ2015,LMZ2018,TY2018}, or modify
the iteration by correction, proximal, or back-substitution steps
\cite{HTY2012,HXY2023,HY2018,LST2015}.
None of these results
establishes convergence from \(A_3=I_m\) alone, leaving the identity-slack case unresolved: the existing literature gives neither a general convergence proof nor a counterexample for this subclass.
Professor Bingsheng He from Nanjing University also  mentioned this open problem in his talk \cite{hebma}.

\subsection{Main Results}
\label{sec:main-results}

In this paper, we give a negative answer to the above open problem through an AI-assisted research process, i.e., direct three-block ADMM may fail even when the first two blocks are strongly convex quadratics.
The main contributions are as follows:
\begin{itemize}
\item {\bf{Mathematically}}.
An exact rational period-\(66\) identity-slack counterexample (Theorem~\ref{thm:main}); problem-dependent multiplier-step convergence and the exclusion of any class-uniform relative step (Section~\ref{sec:relaxation}); and a locally attracting period-\(23\) certificate (Section~\ref{sec:route-comparison}).
\item {\bf{AI-assisted research process}}.
We document how human-guided reasoning-and-coding agents contributed to representation discovery, structured counterexample construction, exact certification, and theorem continuation.  Floating-point divergence or cycling is treated only as a clue; theorem-level claims require exact certificates.  Prompts, workspace records, human interventions, and verification artifacts are archived in Appendix~\ref{app:discovery-provenance}.
\end{itemize}

\subsection{Paper Organization}
\label{sec:organization}

Section~\ref{sec:counterexample} presents the period-\(66\) counterexample
constructed on the Codex (GPT 5.6 Sol) route, together with its analysis and exact
certification; Section~\ref{sec:relaxation} then studies multiplier
relaxation, including when convergence is restored and when it is not.
Section~\ref{sec:route-comparison} presents the period-\(23\) attracting
counterexample from the Kimi Code (Kimi K3) route and compares the different discovery
mechanisms of the two AI-assisted research routes.
Section~\ref{sec:conclusion} summarizes the main findings and discusses
implications for AI for optimization; the evidence workflow, prompts,
interventions, and verification records are collected in
Appendix~\ref{app:discovery-provenance}.

\section{A Period-66 Counterexample}
\label{sec:counterexample}

This section presents a rational period-\(66\) counterexample and its
piecewise-affine certificate.  The construction proceeds through four
mathematical steps: rejecting fixed-branch instability as insufficient without
a realizable projection itinerary; deriving the signed piecewise-affine
realization on \(s=(y,q)\), \(q=z+\lambda\) (Section~\ref{sec:reduction});
shifting the search from \((Q_1,Q_2)\) to resolvents \((M,N)\) and reachable
active-set itineraries; and reconstructing the candidate over \(\mathbb Q\)
with exact checks of closure, branch admissibility, the original ADMM
iteration, and minimality.  AI-assisted process records are collected in
Appendix~\ref{app:discovery-provenance}.

\subsection{Problem Data and Main Theorem}
\label{sec:instance}

Set \(A=B=I_2\), \(\beta=1\), and consider
\begin{equation}
\label{eq:counterexample-qp}
  \begin{aligned}
    \min_{x,y,z\in\R^2}\quad
    &\frac12x^\top Q_1x+\frac12y^\top Q_2y+\ind_{\R_+^2}(z)\\
    \text{s.t.}\quad &x+y+z=\bar b.
  \end{aligned}
\end{equation}
Define
\[
  M=(Q_1+I_2)^{-1},\qquad N=(Q_2+I_2)^{-1}.
\]
Then algorithm \eqref{eq:direct-admm} becomes
\begin{subequations}
\label{eq:specialized-admm}
\begin{align}
  x^{k+1}
  &=M(\bar b-y^k-z^k+\lambda^k),\label{eq:specialized-x}\\
  y^{k+1}
  &=N(\bar b-x^{k+1}-z^k+\lambda^k),\label{eq:specialized-y}\\
  z^{k+1}
  &=\bigl[\bar b-x^{k+1}-y^{k+1}+\lambda^k\bigr]_+,
    \label{eq:specialized-z}\\
  \lambda^{k+1}
  &=\lambda^k-
    (x^{k+1}+y^{k+1}+z^{k+1}-\bar b).
    \label{eq:specialized-lambda}
\end{align}
\end{subequations}
Let
\begin{equation}
\label{eq:short-data}
  \varepsilon=\frac1{1000},\qquad
  \mu=\frac{8957}{10000},\qquad
  \nu=\frac{999}{1000},\qquad
  d_1=(-1,20)^\top,\qquad d_2=(-1,10)^\top,
\end{equation}
and define
\begin{equation}
\label{eq:MNQ-data}
\begin{aligned}
  M&=\varepsilon I_2+(\mu-\varepsilon)
       \frac{d_1d_1^\top}{d_1^\top d_1},
  &Q_1&=M^{-1}-I_2,\\
  N&=\varepsilon I_2+(\nu-\varepsilon)
       \frac{d_2d_2^\top}{d_2^\top d_2},
  &Q_2&=N^{-1}-I_2.
\end{aligned}
\end{equation}

To define the right-hand side, prescribe a KKT point.  Choose
\[
  z^\star=(0,1)^\top,\qquad \lambda^\star=(-1,0)^\top,
\]
and set
\begin{equation}
\label{eq:b-data}
  x^\star=Q_1^{-1}\lambda^\star,\qquad
  y^\star=Q_2^{-1}\lambda^\star,\qquad
  \bar b=x^\star+y^\star+z^\star.
\end{equation}
By construction,
\[
  Q_1x^\star=\lambda^\star,\qquad
  Q_2y^\star=\lambda^\star,\qquad
  x^\star+y^\star+z^\star=\bar b,\qquad
  \lambda^\star\in N_{\R_+^2}(z^\star),
\]
so \((x^\star,y^\star,z^\star,\lambda^\star)\) satisfies the KKT conditions
\eqref{eq:kkt}.  Every entry of \(Q_1,Q_2,\bar b\) is rational.  Moreover,
\begin{equation}
\label{eq:Q-spectra-short}
  \sigma(Q_1)=\left\{999,\frac{1043}{8957}\right\},
  \qquad
  \sigma(Q_2)=\left\{999,\frac1{999}\right\},
\end{equation}
where \(\sigma(Q)\) denotes the spectrum.  Thus \(Q_1,Q_2\succ0\).
Strong convexity in \((x,y)\) and the constraint determine a unique primal
solution, while \(Q_1x^\star=\lambda^\star\) fixes the multiplier.  Hence the
KKT point is unique.  All three ADMM subproblems also have unique minimizers,
so the iteration is single-valued.

\begin{theorem}[Exact rational period-\(66\) counterexample]
\label{thm:main}
For the rational problem \eqref{eq:counterexample-qp}--\eqref{eq:b-data},
there is a rational initialization for which the sequence generated by the
unmodified direct three-block ADMM \eqref{eq:direct-admm} is bounded, non-KKT,
and periodic with minimal period \(66\).  Its strict projection-sign sequence
is
\begin{equation}
\label{eq:mask-word}
  \mathcal W=(00)^2(01)^{64},
\end{equation}
where the \(i\)-th bit is \(1\) if the corresponding projection input is
positive and \(0\) if it is negative.
All \(132\) projection inequalities are strict
with a common margin greater than \(10^{-3}\).
Therefore an identity third constraint block does not guarantee unconditional global convergence of direct three-block ADMM.
\end{theorem}

\subsection{Orbit Geometry}
\label{sec:orbit-geometry}

Figure~\ref{fig:cycle-geometry} is a decimal visualization of the exact
rational sequence in Theorem~\ref{thm:main}.  The values
\(q^k=z^k+\lambda^k\) form a closed loop that does not contain the unique KKT
value \(q^\star=(-1,1)\).

\begin{figure}[H]
\centering
\includegraphics[width=0.92\textwidth]{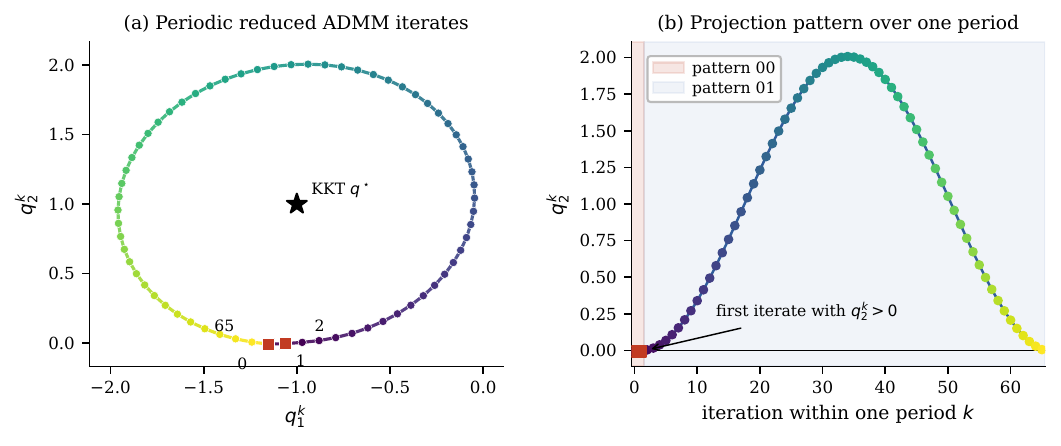}
\caption{Decimal visualization of the exact rational period-\(66\) sequence.
Red squares mark the two \(00\) iterates; the remaining iterates have pattern
\(01\).  The KKT coordinate \(q^\star=(-1,1)\) is not on the periodic
sequence.}
\label{fig:cycle-geometry}
\end{figure}

Numerically, the dominant \(01\)-branch pair has modulus
\(1.00018\) and angle \(0.0951955\approx2\pi/66\), while
\(\rho(T_{00})\approx0.97058\).  Thus \(64\) near-rotational \(01\)
steps followed by two resetting \(00\) steps explain the word
\((00)^2(01)^{64}\).  This is only intuition; the proof follows.

\subsection{Exact Period Certificate}
\label{sec:reduction}

We now reduce the ADMM iteration to verify the periodic sequence in
Theorem~\ref{thm:main}.
For quadratic programs, local ADMM behavior and parameter selection can often
be reduced to linear or affine iterations after active-set identification
\cite{B2013,GTSJ2015,LFP2017}.  Here we instead compose an entire
finite sign itinerary in exact rational arithmetic.

Define the orthant-projection input
\begin{equation}
\label{eq:q-definition}
  q^{k+1}=\bar b-x^{k+1}-y^{k+1}+\lambda^k.
\end{equation}
The \(z\)- and multiplier updates give
\begin{equation}
\label{eq:projection-identity}
  z^{k+1}=[q^{k+1}]_+,\qquad
  \lambda^{k+1}=[q^{k+1}]_-,\qquad
  q^{k+1}=z^{k+1}+\lambda^{k+1}.
\end{equation}
For \(a\in\R^2\), write
\(([a]_+)_i=\max\{a_i,0\}\), \(([a]_-)_i=\min\{a_i,0\}\), and
\((|a|)_i=|a_i|\) for \(i=1,2\).
Thus \eqref{eq:projection-identity} shows that the single vector \(q^{k+1}\)
uniquely determines both \(z^{k+1}\) and \(\lambda^{k+1}\).  Since
\(z^{k+1}=\Pi_{\R_+^2}(q^{k+1})\), the projection optimality condition gives
\[
  \lambda^{k+1}=q^{k+1}-z^{k+1}
  \in N_{\R_+^2}(z^{k+1}).
\]
We index the recurrence after the projection--multiplier update and set
\[
  s^k=((y^k)^\top,(q^k)^\top)^\top\in\R^4.
\]
Since \(z-\lambda=|q|\), the eliminated variables are recovered by
\begin{equation}
\label{eq:raw-reconstruction}
  z=[q]_+,\qquad \lambda=[q]_-,\qquad
  x^+=M(\bar b-y-z+\lambda)=M(\bar b-y-|q|).
\end{equation}
Define \(p:=\bar b-x^+-z+\lambda\) and \(g:=(I_2-M)\bar b\).  Substitution
gives the reduced recurrence
\begin{subequations}
\label{eq:signed-recurrence}
\begin{align}
  p&=My-(I_2-M)|q|+g,\label{eq:signed-p}\\
  y^+&=Np,\label{eq:signed-y}\\
  q^+&=\bar b-x^+-y^++\lambda
       =p+z-y^+=(I_2-N)p+[q]_+.\label{eq:signed-q}
\end{align}
\end{subequations}
Together with \eqref{eq:raw-reconstruction}, this recurrence is equivalent to
the original ADMM after the first projection--multiplier update.

For a strict projection pattern
\(D=\diag(\mathbf 1_{\{q_1>0\}},\mathbf 1_{\{q_2>0\}})\), one has
\([q]_+=Dq\) and \(|q|=S_Dq\), where \(S_D=2D-I_2\).  Hence, on the
corresponding polyhedral projection region,
\eqref{eq:signed-recurrence} is the affine update
\begin{equation}
\label{eq:branch-map}
\begin{aligned}
  s^+&=T_Ds+h,\\
  T_D&=
  \begin{pmatrix}
    NM&-N(I_2-M)S_D\\
    (I_2-N)M&D-(I_2-N)(I_2-M)S_D
  \end{pmatrix},
  & h&=\binom{Ng}{(I_2-N)g}.
\end{aligned}
\end{equation}
Only two projection patterns are needed:
\(D_{00}=\diag(0,0)\) and \(D_{01}=\diag(0,1)\).  Write
\(T_{00}:=T_{D_{00}}\) and \(T_{01}:=T_{D_{01}}\), and represent each affine
update in homogeneous coordinates by
\begin{equation}
\label{eq:affine-lift}
  L_\omega=\begin{pmatrix}T_\omega&h\\0&1\end{pmatrix}
  \in\R^{5\times5},
  \qquad \omega\in\{00,01\}.
\end{equation}
Composing the affine updates in the order prescribed by
\eqref{eq:mask-word} gives
\begin{equation}
\label{eq:period-map}
  L_{01}^{64}L_{00}^2=
  \begin{pmatrix}\mathcal P&a\\0&1\end{pmatrix}.
\end{equation}
Here \(\mathcal P\in\Q^{4\times4}\) and \(a\in\Q^4\).  Exact rational
elimination verifies \(\det(I_4-\mathcal P)\neq0\).  Hence
\(s^{66}=s^0\) has the unique rational solution
\begin{equation}
\label{eq:period-initialization}
  s^0=(I_4-\mathcal P)^{-1}a\in\Q^4.
\end{equation}
Thus the cycle is an affine fixed point selected by the nonzero offset \(a\),
which comes from \(\bar b\), rather than a unit-eigenvalue orbit of
\(\mathcal P\) or an effect of linear objective terms.
Prescribing a branch itinerary and then solving the fixed-point equation of its
return map is related to constructive counterexample methodology
\cite{GDT2023}; the projection admissibility of every step must still be
verified here.

It remains to verify that the fixed point
\eqref{eq:period-initialization} follows the prescribed projection branches.
Exact rational iteration of \eqref{eq:signed-recurrence} gives
\begin{equation}
\label{eq:strict-signs}
  q_1^k<0\quad(0\leq k<66),\qquad
  q_2^k<0\quad(k=0,1),\qquad
  q_2^k>0\quad(2\leq k<66).
\end{equation}
More precisely, if \(\delta_k=0\) for \(k=0,1\) and \(\delta_k=1\) otherwise,
then
\begin{equation}
\label{eq:strict-margin}
  \min_{0\leq k<66}
  \min\{-q_1^k,(2\delta_k-1)q_2^k\}
  >\frac1{1000}.
\end{equation}
Hence every branch is admissible and the orthant projection realizes
\((00)^2(01)^{64}\).

Equation~\eqref{eq:raw-reconstruction} recovers the full ADMM sequence.  Define
\[
  x^0=M(\bar b-y^{65}-|q^{65}|).
\]
Together with \(z^0=[q^0]_+\) and \(\lambda^0=[q^0]_-\), this gives
\((x^0,y^0,z^0,\lambda^0)\in\Q^8\), and the step-\(65\) update returns the
full state.
If the ADMM sequence had a proper subperiod, its projection-sign sequence
would be a proper repetition.  The only possible repetition factor is two,
because the sequence has length \(66\) and contains exactly two occurrences
of \(00\).  Its two length-\(33\) halves, however, contain different numbers
of \(00\) symbols.  The first return therefore occurs at step \(66\).
The periodic sequence cannot contain the unique KKT point, which is a fixed
point of this single-valued ADMM iteration.  It is consequently non-KKT,
nonconvergent, and bounded.  This proves the theorem.

The strict projection margin and nonsingularity of the return equation also
make the cycle persist under sufficiently small perturbations of the problem
data; see Corollary~\ref{cor:robustness}.

\section{Multiplier Relaxation}
\label{sec:relaxation}

With the period-\(66\) certificate established, a natural next question is
whether a smaller multiplier step restores convergence on the same instance.
We kept the rational QP, \(\beta\), and initialization fixed, and had the
Codex route explore only \(\tau\).  The resulting branch maps, candidate
regimes, and certificate proposals were then checked by exact replay; the
authors separated the claims into local KKT stability, convergence from the
specified initialization, and a problem-dependent class-level extension
(Appendix~\ref{app:discovery-provenance}).

Keep the primal penalty \(\beta>0\) and replace only the multiplier update by
\begin{equation}
\label{eq:relaxed-update}
  \lambda^{k+1}=\lambda^k-\tau
  (Ax^{k+1}+By^{k+1}+z^{k+1}-b),
  \qquad \tau>0.
\end{equation}
The standard method corresponds to \(\tau=\beta\).  In the counterexample
below \(\beta=1\), so the certified choice \(\tau\approx1/2\) is a substantial
multiplier relaxation, not a small perturbation of the standard step.  We
first state a class-level result and its limitation, and then give sharper
ranges for the period-\(66\) instance.

\subsection{Small-Step Convergence and Its Limits}

The small-dual-step contraction principle is due to Hong and Luo
\cite[Theorem~3.1 and (3.16)--(3.17)]{HL2017}.  Their argument uses primal
and dual error bounds obtained under compact-polyhedral assumptions.  The
contribution needed here is to verify those bounds globally for the noncompact
slack-last model under global smooth strong convexity and the full-row-rank
condition on \([A\ B]\).  Thus the next proposition is an
assumption-replacement specialization of the Hong--Luo framework, rather than
a new general small-step principle.

\begin{proposition}[Hong--Luo-type slack-last specialization]
\label{prop:general-small-dual-step}
In \eqref{eq:slack-problem}, suppose that
\(F:\R^{n_x}\to\R\) and \(G:\R^{n_y}\to\R\) are continuously
differentiable, respectively \(\mu_F\)- and \(\mu_G\)-strongly convex,
and have \(L_F\)- and \(L_G\)-Lipschitz gradients, where
\(\mu_F,\mu_G>0\) and \(L_F,L_G<\infty\).
Assume also that \([A\ B]\) has full row rank.  For every fixed
\(\beta>0\), there exists
a problem-dependent \(\bar\tau>0\) such that, for every
\(0<\tau<\bar\tau\), the direct \(x\to y\to z\to\lambda\) iteration
with \eqref{eq:relaxed-update} converges R-linearly from every finite
initialization satisfying \(z^0\in\R_+^m\) to the unique KKT point.
\end{proposition}

\begin{proof}
Let \(u=(x,y,z)\), let \(\bar u(\lambda)\) minimize the augmented
Lagrangian at fixed \(\lambda\), and set
\(d_\beta(\lambda)=\min_u\mathcal L_\beta(u;\lambda)\).  Consider
\[
  V^k=
  \mathcal L_\beta(u^{k+1};\lambda^k)-d_\beta(\lambda^k)
  +d_\beta(\lambda^\star)-d_\beta(\lambda^k),
\]
and write
\(S_k=\|u^{k+1}-u^k\|^2\) and
\(R_k=\|\nabla d_\beta(\lambda^k)\|^2\).
Combining the Hong--Luo primal--dual gap calculation with the global primal
and dual error bounds verified here for the noncompact slack-last model gives
finite positive constants
\(\gamma,C,C_p,C_d\), depending only on the problem data and \(\beta\), such
that
\begin{equation}
\label{eq:gap-descent-overview}
  V^k-V^{k-1}
  \leq-(\gamma-\tau C)S_k-\tau R_k,
  \qquad
  V^k\leq C_pS_k+C_dR_k.
\end{equation}
Hence \(0<\tau<\bar\tau:=\gamma/C\) implies
\[
  V^k\leq\rho V^{k-1},\qquad
  \rho=\left(1+
  \min\{(\gamma-\tau C)/C_p,\tau/C_d\}\right)^{-1}<1.
\]
The iterate error is bounded by a constant multiple of \(V^k\), which gives
R-linear convergence.  The full derivation is given in
Appendix~\ref{app:general-small-dual-step-proof}.
\end{proof}

The threshold in Proposition~\ref{prop:general-small-dual-step} cannot be
chosen independently of the problem; the obstruction already occurs in the
strongly convex quadratic subclass.  With
\(\vartheta=\tau/\beta\), the following result gives a quantitative
obstruction.

\begin{theorem}[No problem-independent positive relative-step interval]
\label{thm:no-universal-dual-step}
For every
\[
  0<\vartheta\leq\frac{80}{119},
\]
there is an \(m=3\) slack-last QP with \(Q_1,Q_2\succ0\) and
\([A\ B]\) of full row rank, together with a finite feasible
initialization, for which direct \(x\to y\to z\to\lambda\) ADMM with
relative multiplier step \(\vartheta\) generates a bounded nonconvergent
sequence.  Consequently, there is no \(\bar\vartheta>0\) such that, for every problem in the class, the method converges for every \(0<\vartheta<\bar\vartheta\).
\end{theorem}

\begin{proof}
Fix \(\vartheta\).  Appendices~\ref{app:no-universal-dual-step-proof}
and~\ref{app:universal-step-obstruction} construct a continuous
one-parameter branch matrix \(U(\alpha):=U_D(r,\alpha)\), where
\(D=\diag(1,0,0)\) and
\(r=\sqrt{7\vartheta/(80-7\vartheta)}\).  The branch is unstable at the
degenerate endpoint \(\alpha=0\), while the other endpoint
\(\alpha=\sqrt2-1\) is strictly stable.  Continuity gives an interior
\(\alpha_c\) with \(\rho(U(\alpha_c))=1\); at this parameter the associated
problem is already strongly convex and \([A\ B]\) has full row rank.
Uniqueness of the KKT point excludes the eigenvalue \(+1\), so the remaining
unit-circle mode generates a bounded nonconvergent error sequence.
Finally, the strict projection margin keeps every sufficiently small
perturbation on the same affine branch; the appendix also constructs a
feasible initialization and verifies boundedness of the full ADMM orbit.
\end{proof}

When \(\bar\tau\) is unavailable, the same global error bounds yield an
exact look-ahead backtracking rule with one additional trial primal sweep.
We make no model-independent novelty claim for this rule; it is recorded as a
consequence of the slack-last bounds rather than a primary contribution.

\begin{corollary}[General adaptive dual step with backtracking]
\label{thm:adaptive-dual-step}
Under the assumptions of Proposition~\ref{prop:general-small-dual-step},
write \(u=(x,y,z)\), \(E=[A,B,I_m]\), and set
\[
\begin{aligned}
  \kappa_p&:=
    \frac{1+\max\{L_F,L_G\}+\beta\|E\|^2}
    {\lambda_{\min}\!\left(
      \diag(\mu_FI,\mu_GI,0)+\beta E^\top E\right)},\\
  \operatorname{PG}_\lambda(u)
    &:=u-\operatorname{prox}_h\!\left(u-\nabla s_\lambda(u)\right),
    \qquad B_E(v):=\|E\|^2\kappa_p^2\|v\|^2,
\end{aligned}
\]
where \(h(u)=\iota_{\R_+^m}(z)\) and \(s_\lambda\) is the smooth
part of \(\mathcal L_\beta(\cdot;\lambda)\), while
\(\mathcal G_\lambda\) denotes one exact \(x\to y\to z\)
Gauss--Seidel primal sweep at fixed multiplier.
Fix \(\chi,\delta\in(0,1)\) and \(\tau_{\max}>0\).
Algorithm~\ref{alg:adaptive-dual-step} accepts a trial only when its
\((\widehat v,\widehat D)\) satisfy
\begin{equation}
\label{eq:adaptive-gate}
  \tau B_E(\widehat v)\leq(1-\chi)\widehat D.
\end{equation}
Under exact block solves and exact evaluation of the acceptance test,
backtracking terminates finitely, the accepted steps have a positive lower
bound independent of \(k\), and
\(\{(u^{k+1},\lambda^k)\}_{k\geq1}\) converges globally R-linearly.
\end{corollary}

\begin{algorithm}[H]
\caption{Backtracking dual-step ADMM with an exact trial primal sweep}
\label{alg:adaptive-dual-step}
\begin{algorithmic}[1]
\Require
  \(u^{\mathrm{init}}
    =(x^{\mathrm{init}},y^{\mathrm{init}},z^{\mathrm{init}})\)
  with \(z^{\mathrm{init}}\in\R_+^m\),
  \(\lambda^0\in\R^m\);
  \(\chi,\delta\in(0,1)\), \(\tau_{\max}>0\)
\State \(u^1\gets\mathcal G_{\lambda^0}(u^{\mathrm{init}})\)
\For{\(k=1,2,\ldots\)}
  \State \(r^k\gets Eu^k-b\)
  \For{\(j=0,1,2,\ldots\)}
    \State \(\tau\gets\tau_{\max}\delta^j\);
      \(\widehat\lambda\gets\lambda^{k-1}-\tau r^k\);
      \(\widehat u\gets\mathcal G_{\widehat\lambda}(u^k)\)
    \State \(\widehat v\gets\operatorname{PG}_{\widehat\lambda}(u^k)\);
      \(\widehat D\gets
      \mathcal L_\beta(u^k;\widehat\lambda)
      -\mathcal L_\beta(\widehat u;\widehat\lambda)\)
    \If{\(\tau B_E(\widehat v)\leq(1-\chi)\widehat D\)}
      \State \((\tau_k,\lambda^k,u^{k+1})
        \gets(\tau,\widehat\lambda,\widehat u)\);
        \textbf{exit} inner loop
    \Else
      \State retain \((u^k,\lambda^{k-1})\)
    \EndIf
  \EndFor
\EndFor
\end{algorithmic}
\end{algorithm}

If
\[
  F(x)=\tfrac12x^\top Q_1x+c_1^\top x,\qquad
  G(y)=\tfrac12y^\top Q_2y+c_2^\top y,\qquad Q_1,Q_2\succ0,
\]
set
\[
  K=\diag(Q_1,Q_2,0)+\beta E^\top E,\qquad
  B_K(v):=\beta^{-1}v^\top(K^{-1}+2I+K)v.
\]
Then \(B_K(v)\) is also an upper bound for
\(\|E(u-\bar u(\lambda))\|^2\), so one may replace
\eqref{eq:adaptive-gate} by
\begin{equation}
\label{eq:adaptive-k-gate}
  \tau B_K(\widehat v)\leq(1-\chi)\widehat D,
\end{equation}
with the same finite-backtracking and convergence conclusions.  The proof of
Corollary~\ref{thm:adaptive-dual-step} and of this quadratic refinement is
given in Appendix~\ref{app:adaptive-dual-step-proof}.

\subsection{Stability of the Period-66 Example}

For the rational QP \eqref{eq:counterexample-qp}, set \(\beta=1\) and use
the six-dimensional essential state \(w=(y,z,\lambda)\in\R^6\).  On a strict projection branch
with mask \(D\), one complete update is affine:
\begin{equation}
\label{eq:relaxed-affine-map}
  w^+=T_D(\tau)w+a_D(\tau),
  \qquad T_D(\tau)=T_D^{(0)}+\tau T_D^{(1)}.
\end{equation}
Here \(\tau\) enters only the multiplier block row.  The explicit block
matrices are given in Appendix~\ref{app:period66-branch-map}.  At the unique
KKT point, \(q^\star=(-1,1)\), so the local branch is \(D_{01}\).
Figure~\ref{fig:tau-regimes} separates the three scopes below.

\begin{figure}[H]
\centering
\includegraphics[width=0.95\textwidth]{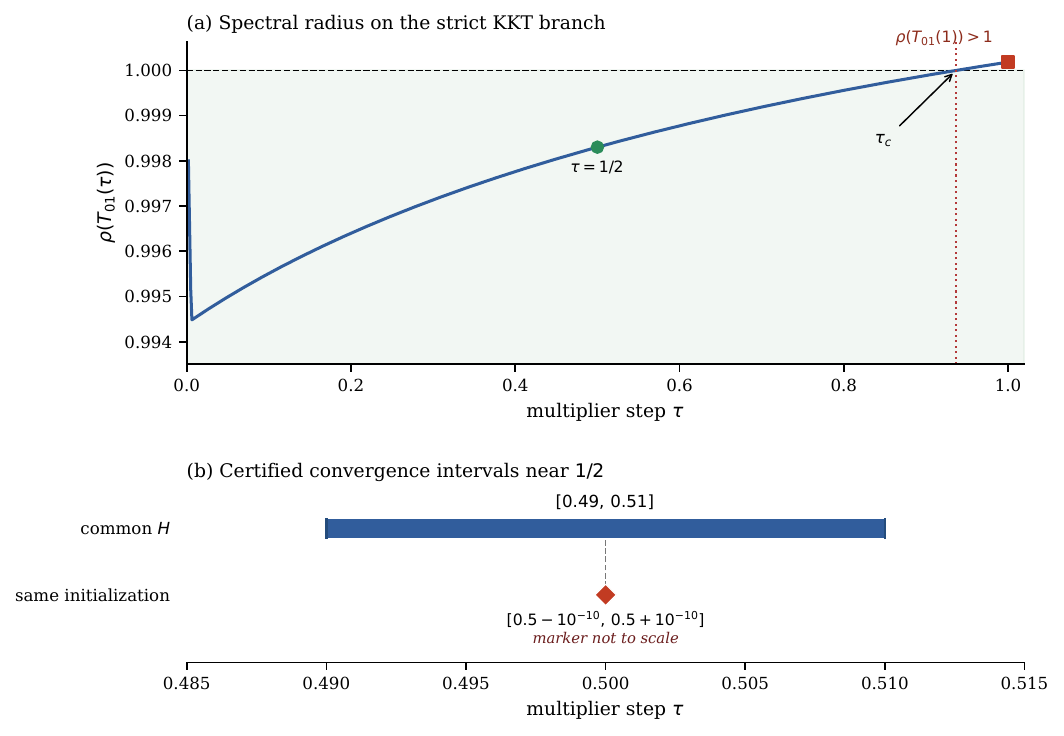}
\caption{Certified dual-step ranges for the fixed rational instance:
pointwise local attraction on the strict \(D_{01}\) branch, a uniform
common-Lyapunov interval, and convergence of the initialization in
Theorem~\ref{thm:main}.}
\label{fig:tau-regimes}
\end{figure}
For \(H\succ0\), write
\[
  \|e\|_H:=(e^\top He)^{1/2}.
\]
For a fixed \(\tau\), we say that the KKT state \(w^\star\) is
\emph{locally Q-linearly attracting} if there exist a neighborhood
\(\mathcal U_\tau\) of \(w^\star\), a norm \(\|\cdot\|_\tau\), and a constant
\(\kappa_\tau\in(0,1)\) such that every \(w^0\in\mathcal U_\tau\) remains in
\(\mathcal U_\tau\) and satisfies
\[
  \|w^{k+1}-w^\star\|_\tau
  \leq
  \kappa_\tau\|w^k-w^\star\|_\tau,
  \qquad k\geq0.
\]
\begin{theorem}[Dual step length for the period-\(66\) instance]
\label{thm:relaxation}
For the rational QP \eqref{eq:counterexample-qp} and
\eqref{eq:relaxed-update}:
\begin{enumerate}
    \item there exist a rational matrix \(H\succ0\), a radius \(r>0\), and a
  constant \(\kappa\in(0,1)\), all independent of
  \(\tau\in[49/100,51/100]\), such that
  \[
    \|w^0-w^\star\|_H\leq r
    \quad\Longrightarrow\quad
    \|w^{k+1}-w^\star\|_H
    \leq
    \kappa\|w^k-w^\star\|_H
  \]
  for every \(k\geq0\).  In particular, the KKT point is uniformly locally
  Q-linearly attracting throughout
  \[
    \frac{49}{100}\leq\tau\leq\frac{51}{100};
  \]
  \item the initialization of Theorem~\ref{thm:main} converges to that KKT
  point whenever
  \[
    \left|\tau-\frac12\right|\leq10^{-10};
  \]
    \item for each fixed \(\tau\in(0,1)\), the KKT state \(w^\star\) is locally
      Q-linearly attracting if and only if
      \[
        0<\tau<\tau_c,
      \]
      where
      \[
        0.9366061114<\tau_c<0.9366061115.
      \]
      Equivalently, \(T_{01}(\tau)\) is Schur stable exactly for
      \(0<\tau<\tau_c\).  For every fixed \(\tau\) in this interval, the
      neighborhood, Lyapunov metric, and contraction factor may depend on
      \(\tau\).
\end{enumerate}
\end{theorem}

\begin{proof}
At \(\tau=1/2\), an exact rational Lyapunov equation produces a common metric.
Exact endpoint tests and a quadratic chord identity extend contraction to
\([0.49,0.51]\); strict complementarity supplies a projection-safe
neighborhood.  For the specified period-\(66\) initialization, an exact
\(232\)-step interval enclosure for
\(|\tau-1/2|\leq10^{-10}\) enters that neighborhood.  Finally, exact
Schur recursion and a Sturm count identify the unique boundary
\(\tau_c\); inside the strict \(D_{01}\) branch, the resulting Lyapunov
metric and the Banach fixed-point theorem give local linear convergence.
The exact derivation, polynomial, and verification commands are in
Appendices~\ref{app:period66-dual-step-proof}
and~\ref{app:period66-dual-step-certificate}.
\end{proof}

The quantifiers in the three assertions are intentionally different.  The
first is uniform in \(\tau\) but local in the initial state; the second
concerns only the initial state used in Theorem~\ref{thm:main}; and the third
is a pointwise local branch criterion whose metric and neighborhood may depend
on \(\tau\).  Only the second assertion says anything about the former
period-\(66\) initialization.  None of the three assertions gives
arbitrary-initial global convergence over the displayed fixed-instance
ranges.  Proposition~\ref{prop:general-small-dual-step} has a fourth,
class-level quantifier: for every fixed problem satisfying its assumptions,
there is a problem-dependent \(\bar\tau\) that works from every finite
initialization.

\section{A Period-23 Attracting Cycle}
\label{sec:route-comparison}

This section presents a separate Kimi Code K3 route that was not supplied with
the period-\(66\) candidate or its certificate.

\subsection{Exact Period-23 Certificate}
\label{sec:k3-route}

The route received the same mathematical question without the Codex matrices,
projection word, initialization, certificate, or project-specific route
guidance (Appendix~\ref{app:provenance-init}).  It produced an \(m=3\)
non-KKT sequence of minimal period \(23\).  More strongly,
Proposition~\ref{prop:k3-period23} certifies an open invariant ellipsoid of
reduced initializations attracted phasewise to that sequence.  All \(69\)
projection inputs have common strict margin greater than \(1/250\), and an
exact Lyapunov inequality certifies local attraction.  The retained run record
and starting condition are summarized in
Appendix~\ref{app:discovery-provenance}.

The route reduced each fixed projection branch to an affine map on
\(v=(y,t)\), \(t=z+\lambda\), screened projection words with KKT-compatible
QP data, and promoted a candidate only after exact rational replay.  An exact
Lyapunov search then certified local attraction.

Consider the rational \(m=3\) QP with \(\beta=1\) and
\begin{equation}
\label{eq:k3-rational-data}
\begin{gathered}
\widehat A=
\begin{bmatrix}
 5/36&1/60&25/97\\
 1/60&5/77&5/91\\
 25/97&5/91&11/12
\end{bmatrix},
\qquad
\widehat B=
\begin{bmatrix}
 4/61&1/54&5/83\\
 1/54&13/92&23/88\\
 5/83&23/88&53/58
\end{bmatrix},\\[4pt]
\widehat Q_x=
\begin{bmatrix}
 85/93&-1/57&-3/11\\
 -1/57&99/100&-5/86\\
 -3/11&-5/86&7/78
\end{bmatrix},
\qquad
\widehat Q_y=
\begin{bmatrix}
 99/100&-1/51&-3/47\\
 -1/51&72/79&-18/65\\
 -3/47&-18/65&4/43
\end{bmatrix},\\[4pt]
\widehat b=
\begin{bmatrix}-2/17\\-1/14\\56/73\end{bmatrix},
\qquad
\widehat c_1=
\begin{bmatrix}33/98\\-33/25\\-24/55\end{bmatrix},
\qquad
\widehat c_2=
\begin{bmatrix}-31/99\\19/36\\-11/20\end{bmatrix}.
\end{gathered}
\end{equation}
Composing the \(23\) branch maps of word \(\mathcal W_{23}\) yields the return
map \(\Phi_{\rm per}(v)=M_{\rm per}v+c_{\rm per}\) \eqref{eq:k3-return-map}
with exact rational fixed point \(\widehat v^0\in\Q^6\).  The certified
projection-sign word is
\begin{equation}
\label{eq:k3-sign-word}
  \mathcal W_{23}
  =(+,-,+)^5\,(-,+,+)^7\,(-,-,+)^2\,(-,-,-)\,(-,-,+)^8 .
\end{equation}

\begin{proposition}[Exact period-\(23\) certificate for nearby initializations]
\label{prop:k3-period23}
For the rational data \eqref{eq:k3-rational-data},
\(\widehat Q_x,\widehat Q_y\succ0\) and
\(\det(\widehat A)\det(\widehat B)\neq0\).  The direct iteration
\eqref{eq:direct-admm}, with \(\beta=1\), applied to
\begin{equation}
\label{eq:k3-counterexample}
  \begin{aligned}
    \min_{x,y,z\in\R^3}\quad
    &\frac12x^\top\widehat Q_x x+\widehat c_1^\top x
      +\frac12y^\top\widehat Q_y y+\widehat c_2^\top y
      +\ind_{\R_+^3}(z)\\
    \text{s.t.}\quad
    &\widehat A x+\widehat B y+z=\widehat b,
  \end{aligned}
\end{equation}
has the following properties.
\begin{enumerate}[label=\textup{(\roman*)}]
  \item It has a unique primal--dual KKT point.

  \item Direct ADMM has a non-KKT sequence of minimal period \(23\).  Its
  strict projection-sign word is \eqref{eq:k3-sign-word}; its exponents
  denote consecutive repetitions and sum to \(23\).  All
  \(23\times3=69\) projection inputs are separated from zero:
  \[
    \min_{0\leq k<23}\min_{1\leq i\leq3}|\widehat t_i^k|
    >\frac1{250}.
  \]

  \item There is a rational matrix \(P\succ0\) for which
  \begin{equation}
  \label{eq:k3-certified-ellipsoid}
    \mathcal E_{\rm cert}
    :=\left\{\widehat v^0+e:e^\top Pe<\frac1{4000}\right\}
  \end{equation}
  is an open invariant set for the \(23\)-step return map
  \eqref{eq:k3-return-map}.  Every \(v^0\in\mathcal E_{\rm cert}\) follows
  the repeating pattern \(\mathcal W_{23}\) and converges phasewise to the
  period-\(23\) sequence.
\end{enumerate}
Consequently, the reconstructed ADMM sequence is nonconvergent for every
reduced initialization \(v^0\in\mathcal E_{\rm cert}\).
\end{proposition}

The proof, together with the reduced branch maps, the return-map
construction, and the invariant-ellipsoid data, is given in
Appendix~\ref{app:k3-period23}.  Because \(\widehat A\) and \(\widehat B\)
are nonsingular, the example converts to an all-identity instance.

\begin{corollary}[Identity-slack equivalence]
\label{cor:k3-identity-block-equivalence}
Since \(\widehat A\) and \(\widehat B\) are nonsingular, the change of
variables \(u=\widehat A x\), \(w=\widehat B y\) converts
\eqref{eq:k3-counterexample} into an equivalent rational strongly convex
\([I_3,I_3,I_3]\) identity-slack QP.  The transformed Hessians remain
positive definite, so Proposition~\ref{prop:k3-period23} yields a locally
attracting period-\(23\) counterexample in the all-identity model as well.
The change-of-variables formulas defining the equivalent rational
all-identity instance are given in
Appendix~\ref{app:k3-period23}.
\end{corollary}

Figure~\ref{fig:k3-period23-robustness} summarizes the certified word,
projection margin, invariant-ellipsoid slice, and spectrum of the
period map.

\begin{figure}[H]
\centering
\includegraphics[width=0.78\textwidth]{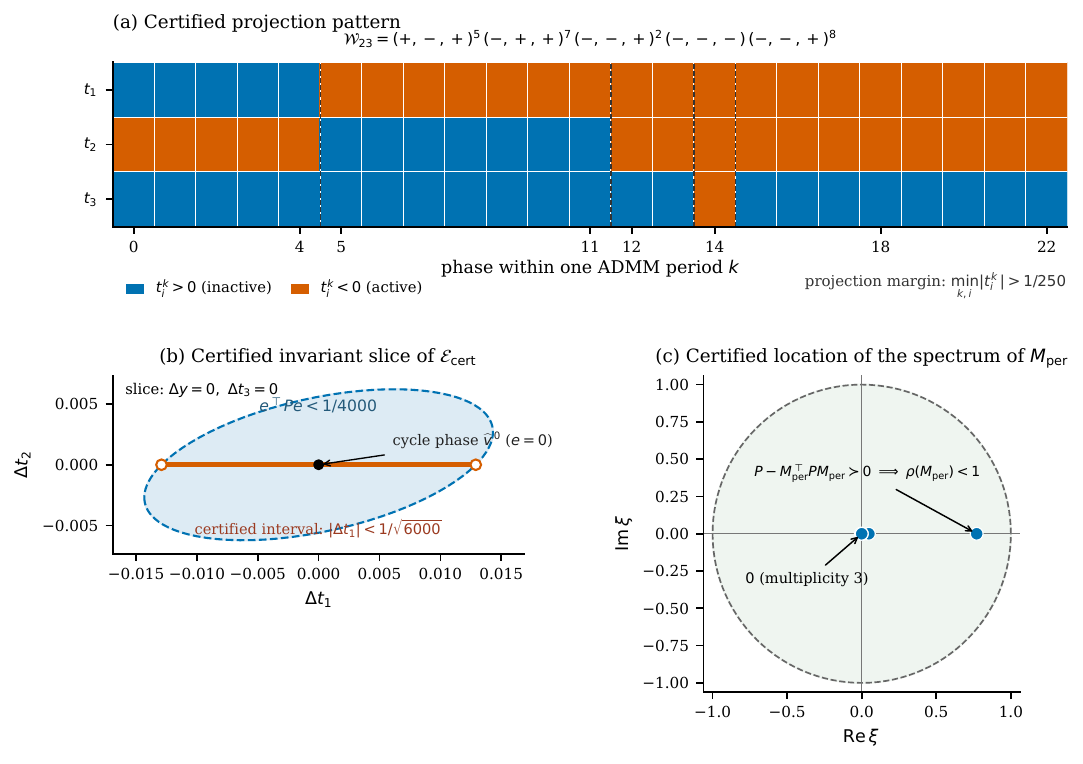}
\caption{Kimi period-\(23\) certificate: projection word and margin;
invariant-ellipsoid slice; spectrum of \(M_{\rm per}\).}
\label{fig:k3-period23-robustness}
\end{figure}

\subsection{Different Discovery Mechanisms of the Two AI-Assisted Routes}
\label{sec:ai-comparison}

This paper compares two research routes that actually occurred and were
fully recorded, rather than the intrinsic capabilities of two base models
under controlled conditions.  To describe the research environment beyond
the model, we write a \emph{research harness} as
\[
  H=(M,W,T,S,A,G,V),
\]
where \(M\) denotes the base model, \(W\) the versioned workspace, \(T\) the
tools, \(S\) the prompts, skills, and instructions, \(A\) the persistent
research artifacts, \(G\) the human decision and claim-promotion gates, and
\(V\) the exact verifiers.
Human direction is a control mechanism inside \(G\), not an independent
component outside \(H\).
A research \emph{route} is one concrete run of a harness configuration
together with the mathematical trajectory it produces.
The present \(H\) is used only to describe the research process for the
identity-slack question in this paper; it does not constitute a
general-purpose AI4Math platform.

The two routes address the same mathematical problem, but under different
harness configurations.
The Codex route ran in a long-maintained project workspace with access to
accumulated failure records, reduced representations, experiment scripts,
and exact certificates, and with iterative human steering.
The Kimi route ran in a frozen blind workspace without access to the
period-\(66\) candidate, certificate, or project history, and received only
a few continue-search instructions.
Both routes share the same evidence principles: floating-point trajectories
and model-generated derivations nominate candidates only; periodicity claims
must pass exact problem validation, strict branch admissibility, raw ADMM
replay, period closure, and minimal-period checks; stability claims further
require the corresponding Lyapunov, Schur, Sturm, or interval tests.
Detailed process dossiers appear in
Appendix~\ref{app:discovery-provenance}.
\begin{figure}[H]
\centering
\resizebox{0.98\textwidth}{!}{%
\begin{tikzpicture}[
  font=\small,
  node distance=4mm,
  stage/.style={
    draw=black!65,
    fill=white,
    rounded corners=2.5pt,
    align=center,
    font=\footnotesize,
    inner xsep=2.5pt,
    inner ysep=4pt,
    minimum height=13mm,
    text width=22mm
  },
  brief/.style={
    stage,
    fill=black!3
  },
  exploratory/.style={
    stage,
    dashed,
    draw=blue!65!black
  },
  certified/.style={
    stage,
    draw=green!45!black,
    line width=0.7pt
  },
  phase band/.style={
    rounded corners=4pt,
    inner xsep=1.5mm,
    inner ysep=2.5mm
  },
  phase title/.style={
    font=\scriptsize,
    align=center
  },
  route/.style={
    draw=black!30,
    fill=black!2,
    rounded corners=2.5pt,
    align=left,
    font=\scriptsize,
    inner sep=3.5pt,
    text width=56mm
  },
  feedback label/.style={
    fill=white,
    rounded corners=2pt,
    align=center,
    font=\scriptsize\itshape,
    inner xsep=3pt,
    inner ysep=1.5pt,
    text width=72mm
  },
  arr/.style={
    -{Stealth[length=2mm]},
    semithick,
    draw=black!70
  },
  feedback/.style={
    -{Stealth[length=2mm]},
    semithick,
    draw=red!55!black
  }
]

\node[brief] (brief)
  {Mathematical brief\\identity-slack};

\node[exploratory, right=of brief] (explore)
  {AI exploration\\reduction\\screening};

\node[exploratory, right=of explore] (candidate)
  {Numerical or\\symbolic candidate};

\node[certified, right=of candidate] (reconstruct)
  {Exact rational\\reconstruction};

\node[certified, right=of reconstruct] (gates)
  {Claim promotion\\exact gates};

\node[certified, right=of gates] (theorem)
  {Theorem-level\\claim};

\node[phase title]
  (explore-title)
  at ($(explore.north)!0.5!(candidate.north)+(0,7.5mm)$)
  {\textbf{Exploration}\\[-1pt]
   \textit{candidate evidence only}};

\node[phase title]
  (certify-title)
  at ($(reconstruct.north)!0.5!(theorem.north)+(0,7.5mm)$)
  {\textbf{Certification}\\[-1pt]
   \textit{exact verification + human audit}};

\begin{scope}[on background layer]
  \node[
    phase band,
    fill=blue!5,
    draw=blue!20,
    fit=(explore)(candidate)(explore-title)
  ] {};

  \node[
    phase band,
    fill=green!5,
    draw=green!20!black,
    fit=(reconstruct)(gates)(theorem)(certify-title)
  ] {};
\end{scope}

\draw[arr] (brief) -- (explore);
\draw[arr] (explore) -- (candidate);
\draw[arr] (candidate) -- (reconstruct);
\draw[arr] (reconstruct) -- (gates);
\draw[arr] (gates) -- (theorem);

\draw[feedback]
  (gates.south)
  .. controls ([yshift=-13mm]gates.south west)
           and ([yshift=-13mm]explore.south east) ..
  node[feedback label,midway,above=1.5mm]
    {Gate failure: revise the search object, candidate,\\
     or evidence requirement}
  (explore.south);

\coordinate (route-mid)
  at ($(candidate)!0.5!(reconstruct)$);

\node[route,anchor=north east]
  at ([xshift=-3mm,yshift=-25mm]route-mid.south)
  {\textbf{Codex route}\\[-1pt]
   Long-running project workspace;\\
   iterative human steering};

\node[route,anchor=north west]
  at ([xshift=3mm,yshift=-25mm]route-mid.south)
  {\textbf{Kimi route}\\[-1pt]
   Frozen blind workspace;\\
   no Codex candidate or certificate};

\end{tikzpicture}%
}
\caption{Candidate-generation and claim-promotion workflow shared by the two
AI-assisted research routes.
Models may propose representations, search objects, and candidates in the
exploration stage, but only conclusions that pass exact verification and a
human claim-scope audit enter the final mathematical record. The two routes
share the same mathematical question and evidence principles, but use
different harness configurations.}
\label{fig:research-loop}
\end{figure}
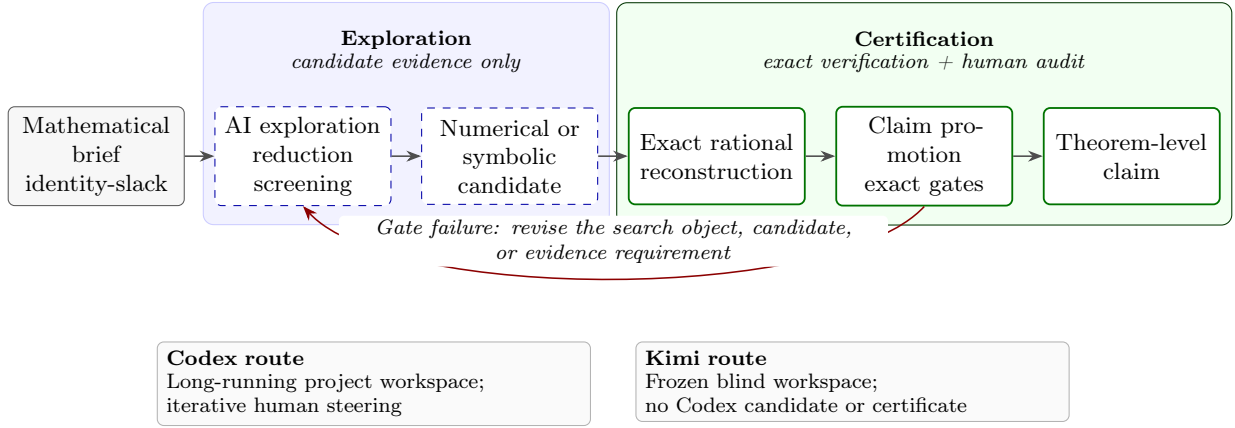

\paragraph{Search space and construction style.}
The two routes first differ markedly in their search objects.
The Codex route reduced the original ADMM iteration to a piecewise-affine
dynamical system and then shifted from searching Hessian data
\((Q_1,Q_2,\bar b)\) to searching resolvent parameters \((M,N)\) and a
realizable active-set itinerary.  Its core object is a switching system
formed by distinct projection branches, culminating in the itinerary
\[
  (00)^2(01)^{64}.
\]
Thus the route is closer to a construction driven by dynamical mechanism:
first identify a near-rotation/reset structure, then recover an optimization
problem that realizes it.

The Kimi route is closer to inverse design of an optimization instance.
It jointly searched over KKT-compatible quadratic-program data
\[
  (\widehat A,\widehat B,\widehat Q_x,\widehat Q_y,
    \widehat c_1,\widehat c_2)
\]
and a realizable projection itinerary \(\mathcal W_{23}\), and used a fixed
point of the period return map to construct an ADMM instance with a
prescribed dynamical property.

\paragraph{State representation and analysis object.}
The Codex route used the four-dimensional signed piecewise-affine state
\[
  s=(y,q),\qquad q=z+\lambda,
\]
whose main role is to expose active-set switching.  The period-\(66\)
itinerary can be interpreted as a comparatively long near-rotation phase,
followed by a short reset through a few \(00\) branches.

The Kimi route used the six-dimensional projector-aligned state
\[
  v=(y,t),\qquad t=z+\lambda,
\]
and analyzed the one-period return map
\[
  \Phi_{\rm per}(v)=M_{\rm per}v+c_{\rm per}
\]
directly.  This representation unifies existence and local stability of the
periodic orbit as questions about a return-map fixed point and contraction.

\paragraph{Certificate goals and dynamical conclusions.}
The primary goal of the Codex certificate is a strict existence proof for a
non-KKT periodic orbit.  Its evidence includes exact period closure, strict
branch admissibility, raw ADMM replay, and minimal-period verification.
The resulting period-\(66\) counterexample moreover retains the same branch
itinerary under small data perturbations.

The Kimi certificate further seeks local attraction of the periodic orbit.
In addition to period closure and strict projection conditions, it
constructs a rational Lyapunov matrix satisfying
\[
  P-M_{\rm per}^{\top}PM_{\rm per}\succ0
\]
and an invariant ellipsoid preserved by the return map.
Thus period-\(23\) nonconvergence is not the phenomenon of a single special
initialization, but occurs on an open set of reduced initializations.

\paragraph{Continuation in optimization theory.}
The Codex route not only produced a counterexample, but also carried the
active-set switching instability into multiplier-relaxation analysis:
after freezing the problem and initialization and varying only the
multiplier step \(\tau\), it further studied problem-dependent small-step
convergence, the impossibility of a class-uniform step, and a fixed-instance
local stability boundary.
The theoretical emphasis of the Kimi route is instead on return-map
contraction and basin stability under reduced-initialization perturbations.

Table~\ref{tab:workflow-comparison} summarizes the main differences between
the two routes in harness configuration, mathematical search objects, and
final certificates.

\begin{table}[H]
\centering
\caption{Comparison of the two realized AI-assisted research routes.}
\label{tab:workflow-comparison}
\small
\setlength{\tabcolsep}{4pt}
\renewcommand{\arraystretch}{1.14}
\begin{tabular}{
  >{\raggedright\arraybackslash}p{0.175\textwidth}
  >{\raggedright\arraybackslash}p{0.365\textwidth}
  >{\raggedright\arraybackslash}p{0.365\textwidth}}
\toprule
\textbf{Aspect}
&
\textbf{Codex / GPT-5.6 Sol route}
&
\textbf{Kimi Code / Kimi K3 route}
\\
\midrule

\textbf{Harness configuration}
&
Long-running project workspace; project-specific ADMM skill and experiment
toolchain; persistent research artifacts; iterative human steering
&
Frozen blind workspace; generic research skills with shell/Python tools;
isolated route state; few continue-search instructions
\\
\midrule

\textbf{Main search object}
&
Resolvent parameters \((M,N)\) and realizable active-set itinerary
\((00)^2(01)^{64}\)
&
KKT-compatible QP data and projection itinerary \(\mathcal W_{23}\)
\\
\midrule

\textbf{Reduced state}
&
\(s=(y,q)\in\R^4\), \(q=z+\lambda\);
direct \(m=2\) all-identity model
&
\(v=(y,t)\in\R^6\), \(t=z+\lambda\);
\(m=3\) model, equivalently convertible to all-identity form
\\
\midrule

\textbf{Construction style}
&
Mechanism-driven: state reduction, branch-spectrum analysis, and
near-rotation/reset itinerary design
&
Inverse design: jointly synthesize problem data, projection itinerary, and
return-map fixed point
\\
\midrule

\textbf{Main certificate}
&
Exact period closure, strict branch checks, raw ADMM replay, and
minimal-period verification
&
Exact period closure, strict projection checks, discrete Lyapunov
inequality, and invariant ellipsoid
\\
\midrule

\textbf{Dynamical conclusion}
&
Minimal period-\(66\) non-KKT orbit; exposes active-set switching
instability
&
Minimal period-\(23\) locally attracting non-KKT orbit; certified open
attracting basin
\\
\midrule

\textbf{Robustness type}
&
Structural persistence under problem-data perturbations
&
Basin stability under reduced-initialization perturbations
\\
\midrule

\textbf{Theorem continuation}
&
Multiplier relaxation, step-size sensitivity, and class-level step
conclusions
&
Return-map contraction and local attracting-basin certification
\\
\midrule

\textbf{Recorded process use}
&
\(\approx 40\)\,h wall-clock;
\(\approx 1.704\times10^9\) recorded tokens
&
\(\approx 9\)\,h wall-clock;
\(\approx 6.327\times10^7\) recorded tokens
\\
\bottomrule
\end{tabular}

\vspace{4pt}
\parbox{0.95\textwidth}{\footnotesize\emph{Scope note.}
The times and token counts in the table record only the process use of two
realized runs.
Because the routes differ in base model, workspace, tools, accumulated
context, token accounting, and human steering, these numbers do not
constitute a controlled model-efficiency comparison.
Full prompts, workspace manifests, intervention records, and verification
commands appear in Appendix~\ref{app:discovery-provenance}.}
\end{table}

The two routes illustrate complementary modes of AI-assisted optimization
research.
The Codex route emphasizes discovery of a hidden state-space structure and
an algorithmic instability mechanism;
the Kimi route emphasizes inverse construction of an optimization instance
with prescribed dynamical behavior, together with certification of its local
stability.
These differences are consistent with the respective harness configurations
and historical paths, suggesting that persistent workspaces, prior
representations, project skills, evidence gates, and human decisions may
jointly shape the mathematical space a model actually explores.

Nevertheless, the present records cannot uniquely attribute the divergence
to the base model, workspace memory, project skills, or human steering
alone.
Identifying their separate effects would require controlled model--harness
ablations under matched tools, budgets, and prompts.
Accordingly, the comparison in this paper should be understood as one
between two realized \emph{model--harness routes}, not as a ranking of
base-model capabilities.

\section{Discussion}
\label{sec:conclusion}

\subsection{What the Results Show}

An identity last coefficient matrix does not by itself restore
convergence of unmodified direct three-block ADMM.  The
period-\(66\) example rules out explanations based on nonconvexity,
nonunique block solves, projection ties, or floating-point drift, and
its itinerary persists under small data perturbations.  The
period-\(23\) example adds a certified open attracting basin and,
by nonsingularity of \(\widehat A\) and \(\widehat B\), an equivalent
all-identity \([I_3,I_3,I_3]\) instance
(Corollary~\ref{cor:k3-identity-block-equivalence}).
Multiplier relaxation further shows that a problem-dependent small
step can restore convergence, but no positive relative step is
class-uniform.

\subsection{AI for Optimization: From Problem Solving to Theory Discovery}

Broadly, AI for optimization concerns not only using AI to solve a given
optimization model, but also using AI in the formation of optimization
models, algorithms, and theory.  Relative to the object of study, this role
can be organized into three interrelated levels.

At the \emph{problem-solving level}, AI may assist formulation and
reformulation, variable and constraint decomposition, algorithm and solver
selection, initialization, parameter configuration, and runtime diagnosis or
intervention.  The usual goal is to improve efficiency, robustness, or
automation on a fixed problem.

At the \emph{algorithm-design level}, AI may further participate in the design
of update rules, relaxation parameters, penalty terms, correction steps, and
adaptive strategies.  The task is no longer only to execute an existing
algorithm, but to diagnose failure modes and propose analyzable
modifications suggested by the discovered structure.  The multiplier-
relaxation results in this paper illustrate this level: after the
period-\(66\) counterexample had been certified, the research continued from
``why the algorithm fails'' to ``which part of the update restores
stability'', and then distinguished problem-dependent convergent steps,
fixed-instance local stability boundaries, and the impossibility of
class-uniform guarantees.

At the \emph{theory-discovery level}, AI may participate in mathematical
representation discovery, structured instance design, counterexample and
extreme-instance construction, theorem conjecture, proof-route search, and
exact certificate generation.  The present work primarily studies this
theory-facing role.  The two routes show that AI can reshape the search space
of optimization research in different ways: 
the Codex route reduced the
original ADMM iteration to a low-dimensional signed piecewise-affine system
and turned counterexample search into the design of resolvent parameters and
reachable active-set itineraries; the Kimi route jointly searched over
KKT-compatible problem data and projection itineraries, turning periodic
counterexample construction into inverse instance design with prescribed
return-map properties.

The AI-for-optimization capabilities illustrated here can therefore be
summarized in the following stages.

\paragraph{Representation discovery.}
The raw iteration of a complicated optimization algorithm often contains
redundant variables, implicit projections, and multilayered dependencies.
AI can search for more analysis-friendly state variables, coordinate changes,
operator splittings, or piecewise structures, thereby transforming the
original problem into a finite-dimensional linear, affine, or operator
dynamical system.  A change of representation affects not only proof
difficulty, but also which mathematical objects can subsequently be searched
and certified.

\paragraph{Structured construction and inverse design.}
AI can convert open questions such as ``does a counterexample exist?'' into
construction problems with explicit parameterizations, branch conditions, and
algebraic constraints.  The goal may be to design an optimization instance on
which an algorithm exhibits a prescribed behavior---for example a periodic
orbit, boundary stability, slow convergence, active-set switching, or
parameter sensitivity.  The same capability can also generate stress-test
instances and robustness benchmarks, not only theoretical counterexamples.

\paragraph{Failure diagnosis and algorithmic intervention.}
When an algorithm exhibits nonconvergence, oscillation, or numerical
instability, AI can help separate sources such as objective geometry,
constraint degeneracy, block ordering, active-set switching, or parameter
choice.  The resulting dynamical structure can then guide algorithmic
modifications, including relaxation, damping, proximal regularization,
correction steps, and adaptive parameter rules.  Counterexample construction
is therefore not only a negative result, but also an input to algorithm
design.

\paragraph{Proof search and exact certification.}
Numerical experiments and floating-point computation may discover candidates,
but they cannot independently carry theorem-level evidence.  AI can organize
symbolic derivation, branch enumeration, rational reconstruction, interval
arithmetic, Lyapunov search, Schur stability tests, and Sturm root counting
into a reproducible certification pipeline.  In this paper, agents nominated
candidate representations, derivations, and certificate designs; period
closure, strict branch conditions, raw ADMM replay, and the corresponding
stability predicates were supported by independently executable exact checks.
Analytic conclusions remain established by the mathematical arguments in the
main text and appendices.

\paragraph{Theorem continuation.}
The role of AI can extend beyond ``finding an answer'' to a single question.
After a candidate or theorem has been certified, representations, code,
failure records, and proof artifacts in a persistent workspace can be reused
to pose neighboring questions---changing a parameter, relaxing an assumption,
seeking a boundary case, or establishing an impossibility result.  The
progression in this paper from the period-\(66\) counterexample to local
multiplier-step stability, convergence from a specified initialization,
problem-dependent global guarantees, and the impossibility of class-uniform
guarantees is an instance of such theorem continuation.

The overall process can be summarized as
\begin{equation*}
\begin{aligned}
  &\text{problem formulation}
  \longrightarrow
  \text{representation discovery}
  \longrightarrow
  \text{candidate or algorithm design}\\
  &\longrightarrow
  \text{analysis}
  \longrightarrow
  \text{exact certification}
  \longrightarrow
  \text{theorem or algorithm continuation}
  \longrightarrow
  \text{knowledge reuse}.
\end{aligned}
\end{equation*}
The final reuse step may include distilling effective state reductions, proof
templates, certificate structures, and experimental workflows into research
skills that can be invoked on other optimization problems.

\paragraph{The role of the research harness.}
These capabilities are not produced by the base model alone; they depend on
the research environment in which the model operates.  A versioned workspace
provides long-term research memory; tools and skills delimit executable
operations; persistent artifacts allow earlier results to be reused; exact
verifiers prescribe the evidentiary standard by which candidates enter the
mathematical record; and human checkpoints control problem choice, claim
scope, literature boundaries, and final acceptance.  Thus a research harness
does more than store research logs: it also shapes the space of
representations, candidates, and proofs that a model can explore.

This paper presents one case study of that research pattern on the
identity-slack ADMM question.  It does not prove that the same capabilities
already transfer reliably to other optimization domains, nor can it separate
the causal contributions of the base model, workspace memory, project skills,
and human steering from the difference between the two routes.  Future work
in AI for optimization may further explore Lyapunov-function construction,
error-bound discovery, splitting design, adaptive update rules,
extreme-instance generation, and reusable proof patterns, and may evaluate
transferability through frozen workspaces, matched tool budgets, and
controlled ablation experiments.

\subsection{Limits and Outlook}

Several mathematical questions remain open.  First, it is still unclear
whether the two-dimensional pure-quadratic subclass admits a shorter
counterexample whose projection branches are all strictly admissible.
Second, quantitative conditions excluding active-set switching instability
are still needed, as is a sharper characterization of the exact global
multiplier-step range for a fixed problem.  Finite-precision computation,
inexact block solves, alternative block update orders, nonorthant cone
constraints, and more general nonsmooth models also require further analysis.

The present certificates rely mainly on exact rational arithmetic, finite
symbolic derivation, and independently executable algebraic and spectral
checks; they have not yet been fully translated into formal proofs inside a
proof assistant.  A natural next step is to encode the critical ingredients---
period closure, branch admissibility, Lyapunov inequalities, and
Schur--Sturm tests---in Lean, Isabelle, or another formal verification
system, thereby narrowing the gap between computer-assisted certificates and
fully formalized proofs.

On the AI-evaluation side, this paper compares two realized
model--harness routes rather than base-model capabilities under controlled
conditions.  Future studies should use frozen workspaces, preregistered
success criteria, matched tool and compute budgets, and controlled human or
model baselines.  Systematic ablations of the base model, workspace memory,
project-specific skills, persistent research artifacts, and human steering
are also needed to identify each component's contribution to representation
discovery, candidate construction, certificate design, and theorem
continuation.  Another important question is whether the state-reduction,
structured-search, and exact-certification patterns developed here transfer
to Douglas--Rachford splitting, primal--dual splitting, and other operator
splitting methods, rather than merely rediscovering or reproducing the ADMM
examples of this paper.

\newcommand{\bysame}{Tianyi Lin, Shiqian Ma, and Shuzhong Zhang}
{\small
\bibliographystyle{amsplain}
\bibliography{references}
}

\appendix

\section{Multiplier-Relaxation Proofs}
\label{app:certificates}

\subsection{Proof of Proposition~\ref{prop:general-small-dual-step}}
\label{app:general-small-dual-step-proof}
{\small
\begin{proof}
Write \(u=(x,y,z)\) and \(E=[A,B,I_m]\).  Decompose the augmented
Lagrangian as
\[
  \mathcal L_\beta(u;\lambda)=s_\lambda(u)+h(u),
\]
where
\[
  s_\lambda(u)
  :=F(x)+G(y)-\lambda^\top(Eu-b)
    +\frac{\beta}{2}\|Eu-b\|^2,
  \qquad
  h(u):=\iota_{\R_+^m}(z).
\]
Set
\[
  d_\beta(\lambda)=\min_u\mathcal L_\beta(u;\lambda).
\]
We first verify uniqueness of the KKT point.  Full row rank of
\([A\ B]\) implies strict feasibility: for any \(z>0\), the equation
\([A\ B](x,y)=b-z\) is solvable.  Strong convexity gives a unique primal
solution.  In the present multiplier convention, the unaugmented negative
dual objective is
\[
  \phi_0(\lambda)
  =F^*(A^\top\lambda)+G^*(B^\top\lambda)
   -b^\top\lambda+\iota_{\R_-^m}(\lambda).
\]
Standard conjugacy results imply that an \(L\)-smooth convex function has a
\(1/L\)-strongly convex conjugate \cite{BC2017}.  Hence \(F^*\) and
\(G^*\) are \(1/L_F\)- and \(1/L_G\)-strongly convex, respectively.
Consequently,
\[
  m_0=
  \lambda_{\min}\!\left(
    \frac{AA^\top}{L_F}+\frac{BB^\top}{L_G}
  \right)>0,
\]
where strict positivity follows from the full-row-rank assumption.
The multiplier is therefore unique.

We next establish the global primal error bound at fixed multiplier.  Set
\[
  K_\mu=\diag(\mu_FI,\mu_GI,0)+\beta E^\top E,\qquad
  \mu_p=\lambda_{\min}(K_\mu)>0,
\]
\[
  L_p=\max\{L_F,L_G\}+\beta\|E\|^2.
\]
For all \(u,v\),
\[
\begin{aligned}
  \langle\nabla s_\lambda(u)-\nabla s_\lambda(v),u-v\rangle
  &\geq (u-v)^\top K_\mu(u-v)
   \geq\mu_p\|u-v\|^2,\\
  \|\nabla s_\lambda(u)-\nabla s_\lambda(v)\|
  &\leq L_p\|u-v\|.
\end{aligned}
\]
Let \(\bar u=\bar u(\lambda)\) denote the unique minimizer of
\(\mathcal L_\beta(\cdot;\lambda)\), and set
\[
  p=\operatorname{prox}_h\bigl(u-\nabla s_\lambda(u)\bigr),
  \qquad
  r=u-p=\operatorname{PG}_\lambda(u).
\]
Proximal optimality gives
\[
  r-\nabla s_\lambda(u)\in\partial h(p),
  \qquad
  -\nabla s_\lambda(\bar u)\in\partial h(\bar u).
\]
By monotonicity of \(\partial h\),
\[
  \left\langle
    \nabla s_\lambda(u)-\nabla s_\lambda(\bar u),
    p-\bar u
  \right\rangle
  \leq
  \langle r,p-\bar u\rangle.
\]
Using the \(\mu_p\)-strong monotonicity and \(L_p\)-Lipschitz continuity of
\(\nabla s_\lambda\), we obtain
\[
\begin{aligned}
  \mu_p\|u-\bar u\|^2
  &\leq
  \left\langle
    \nabla s_\lambda(u)-\nabla s_\lambda(\bar u),
    u-\bar u
  \right\rangle\\
  &=
  \left\langle
    \nabla s_\lambda(u)-\nabla s_\lambda(\bar u),
    p-\bar u
  \right\rangle
  +
  \left\langle
    \nabla s_\lambda(u)-\nabla s_\lambda(\bar u),
    r
  \right\rangle\\
  &\leq
  \langle r,p-\bar u\rangle
  +L_p\|u-\bar u\|\,\|r\|\\
  &\leq
  (1+L_p)\|u-\bar u\|\,\|r\|.
\end{aligned}
\]
Therefore
\[
  \|u-\bar u(\lambda)\|
  \leq
  \kappa_p\|\operatorname{PG}_\lambda(u)\|,
  \qquad
  \kappa_p=\frac{1+L_p}{\mu_p},
\]
globally, with a constant independent of \(\lambda\).

Exact \(x\to y\to z\) block minimization gives
\[
  \gamma=\frac12\min\!\left\{
  \lambda_{\min}(\mu_FI+\beta A^\top A),
  \lambda_{\min}(\mu_GI+\beta B^\top B),\beta
  \right\}>0
\]
and
\[
  \mathcal L_\beta(u^k;\lambda^k)
  -\mathcal L_\beta(u^{k+1};\lambda^k)
  \geq\gamma\|u^{k+1}-u^k\|^2.
\]
Let
\[
  a_x=L_F+\beta\|A^\top A\|,\qquad
  a_y=L_G+\beta\|B^\top B\|,
\]
\[
  \sigma^2=
  a_x^2+\beta^2\|B^\top A\|^2+a_y^2+
  \beta^2\|A\|^2+\beta^2\|B\|^2+(\beta+2)^2.
\]
For
\(\Delta x=x^k-x^{k+1}\), \(\Delta y=y^k-y^{k+1}\), and
\(\Delta z=z^k-z^{k+1}\), write the three blocks of
\(\operatorname{PG}_{\lambda^k}(u^k)\) as
\((\operatorname{PG}_{x,k},\operatorname{PG}_{y,k},
\operatorname{PG}_{z,k})\).  The block optimality conditions and
nonexpansiveness of the orthant projection then give
\[
\begin{aligned}
  \|\operatorname{PG}_{x,k}\|&\leq a_x\|\Delta x\|,\\
  \|\operatorname{PG}_{y,k}\|
    &\leq\beta\|B^\top A\|\|\Delta x\|+a_y\|\Delta y\|,\\
  \|\operatorname{PG}_{z,k}\|&\leq\beta\|A\|\|\Delta x\|
    +\beta\|B\|\|\Delta y\|+(\beta+2)\|\Delta z\|.
\end{aligned}
\]
Consequently,
\[
  \|\operatorname{PG}_{\lambda^k}(u^k)\|
  \leq\sigma\|u^{k+1}-u^k\|.
\]
For the \(x,y\) blocks this uses only global Lipschitz continuity of
\(\nabla F,\nabla G\); for \(z\), it compares the projection fixed-point
identities at the current point and the block minimizer.

Completing the square in the quadratic penalty and applying Fenchel duality
gives the explicit Moreau-envelope representation
\begin{equation}
\label{eq:augmented-dual-moreau}
  -d_\beta(\lambda)
  =
  \min_{\eta\in\R^m}
  \left\{
    \phi_0(\eta)
    +\frac{1}{2\beta}\|\eta-\lambda\|^2
  \right\}.
\end{equation}
Since \(\phi_0\) is \(m_0\)-strongly convex, its Moreau envelope in
\eqref{eq:augmented-dual-moreau} is strongly convex with modulus
\[
  m_d=\frac{m_0}{1+\beta m_0}>0
\]
and has \(1/\beta\)-Lipschitz gradient \cite{BC2017}.
It follows that
\[
  \|\lambda-\lambda^\star\|
  \leq m_d^{-1}\|\nabla d_\beta(\lambda)\|,
  \qquad
  \nabla d_\beta(\lambda)=b-E\bar u(\lambda).
\]

Define
\[
\begin{aligned}
  \Delta_p^k
  &=\mathcal L_\beta(u^{k+1};\lambda^k)-d_\beta(\lambda^k),\\
  \Delta_d^k
  &=d_\beta(\lambda^\star)-d_\beta(\lambda^k),\\
  V^k&=\Delta_p^k+\Delta_d^k .
\end{aligned}
\]
The gap calculation of Hong and Luo
\cite[Theorem~3.1, Lemma~3.1, and (3.16)--(3.17)]{HL2017},
in the present sign convention, yields
\[
  V^k-V^{k-1}
  \leq
  -a\|u^{k+1}-u^k\|^2
  -\tau\|\nabla d_\beta(\lambda^k)\|^2,
  \qquad
  a=\gamma-\tau\|E\|^2\kappa_p^2\sigma^2 .
\]
Thus \(a>0\) whenever
\[
  0<\tau<
  \frac{\gamma}{\|E\|^2\kappa_p^2\sigma^2}.
\]
The cost-to-go estimate on the noncompact orthant can be verified directly.
Set
\(\Delta x=x^k-x^{k+1}\),
\(\Delta y=y^k-y^{k+1}\), and
\(\Delta z=z^k-z^{k+1}\), and define
\[
  H_{\rm GS}
  =\beta
  \begin{pmatrix}
    0&A^\top B&A^\top\\
    0&0&B^\top\\
    0&0&0
  \end{pmatrix},
  \qquad C_p=\frac{\|H_{\rm GS}\|^2}{2\mu_p}.
\]
The three exact block optimality conditions, including the normal-cone
condition for \(z\), give
\[
  w^{k+1}=
  \begin{pmatrix}
    -\beta A^\top(B\Delta y+\Delta z)\\
    -\beta B^\top\Delta z\\
    0
  \end{pmatrix}
  \in\partial_u\mathcal L_\beta(u^{k+1};\lambda^k).
\]
Uniform \(\mu_p\)-strong convexity at fixed multiplier therefore yields
\[
  \Delta_p^k
  \leq\frac{\|w^{k+1}\|^2}{2\mu_p}
  \leq C_p\|u^{k+1}-u^k\|^2.
\]
Moreover, \(q_\beta=-d_\beta\) is \(m_d\)-strongly convex and
\(1/\beta\)-smooth, so one may take
\[
  C_d=\frac{1}{2\beta m_d^2},\qquad
  \Delta_d^k
  \leq C_d\|\nabla d_\beta(\lambda^k)\|^2.
\]
These constants are independent of \(k,\lambda^k\), and the dual step.
Consequently,
\[
  \eta=\min\{a/C_p,\tau/C_d\}>0,\qquad
  \rho=(1+\eta)^{-1}\in(0,1)
\]
imply
\[
  V^k-V^{k-1}\leq-\eta V^k,\qquad
  V^k\leq\rho V^{k-1}.
\]

Finally, strong convexity and Lipschitz continuity of the fixed-multiplier
minimizer give
\[
\begin{aligned}
  \|u^{k+1}-\bar u(\lambda^k)\|^2
  &\leq 2\Delta_p^k/\mu_p,\\
  \|\lambda^k-\lambda^\star\|^2
  &\leq 2\Delta_d^k/m_d,\\
  \|\bar u(\lambda^k)-u^\star\|
  &\leq(\|E\|/\mu_p)\|\lambda^k-\lambda^\star\|.
\end{aligned}
\]
Hence there is a constant \(C>0\) for which
\[
  \|(u^{k+1},\lambda^k)-(u^\star,\lambda^\star)\|
  \leq C\rho^{k/2}.
\]
The assumption \(z^0\in\R_+^m\) in
Proposition~\ref{prop:general-small-dual-step} ensures that the initial
augmented-Lagrangian value is finite.  The \(z\)-update then preserves
\(z^k\in\R_+^m\) for every \(k\geq0\).
The compact-polyhedral assumption in the original Hong--Luo theorem is not
invoked directly here: the global smooth strongly convex structure above
replaces the error bounds for which that assumption is used.
Accordingly, the inherited component is the gap-contraction architecture; the
model-specific component proved here is the verification of the required
global bounds and explicit constants on the noncompact slack-last class.
\end{proof}
}

\subsection{Backtracking Rule}
\label{app:adaptive-dual-step-proof}
{\small
This appendix proves Corollary~\ref{thm:adaptive-dual-step} and the quadratic
refinement stated in Section~\ref{sec:relaxation}.

\paragraph{Proof of Corollary~\ref{thm:adaptive-dual-step}.}
Fix an arbitrary candidate \(\tau\), and abbreviate its trial values by
\[
  \lambda=\widehat\lambda,\qquad
  u=u^k,\qquad u^+=\widehat u,\qquad
  D=\mathcal L_\beta(u;\lambda)-\mathcal L_\beta(u^+;\lambda).
\]
Let
\[
\begin{aligned}
  e&=u-\bar u(\lambda),\\
  v&=\operatorname{PG}_\lambda(u).
\end{aligned}
\]
The global primal error bound in
Proposition~\ref{prop:general-small-dual-step} directly gives
\begin{equation}
\label{eq:adaptive-majorant}
  \|Ee\|^2
  \leq\|E\|^2\|e\|^2
  \leq\|E\|^2\kappa_p^2\|v\|^2
  =B_E(v).
\end{equation}
The trial decrease can be evaluated from the block optimality conditions,
without subtracting two nearby augmented-Lagrangian values.  After the
initial primal sweep, \(z,z^+\in\R_+^m\).  If
\(\Delta x=x-x^+\), \(\Delta y=y-y^+\), and
\[
  q=b-Ax^+-By^++\lambda/\beta,
  \qquad z^+=[q]_+,
\]
and if
\[
  D_F(a,c)=F(a)-F(c)-\langle\nabla F(c),a-c\rangle
\]
with \(D_G\) defined analogously, cancellation of the first-order terms in
the block optimality conditions gives
\begin{equation}
\label{eq:adaptive-block-decrease}
\begin{aligned}
  D={}&
  D_F(x,x^+)+\tfrac\beta2\|A\Delta x\|^2\\
  &+D_G(y,y^+)+\tfrac\beta2\|B\Delta y\|^2\\
  &+\tfrac\beta2
  \bigl(\|z-q\|^2-\|z^+-q\|^2\bigr).
\end{aligned}
\end{equation}

Exact block minimization gives uniform \(\gamma,\sigma>0\) such that
\[
  D\geq\gamma\|u^+-u\|^2,\qquad
  \|v\|\leq\sigma\|u^+-u\|.
\]
Consequently every
\[
  0<\tau\leq
  \tau_{\mathrm{safe},E}
  :=\frac{(1-\chi)\gamma}
  {\|E\|^2\kappa_p^2\sigma^2}
\]
satisfies \eqref{eq:adaptive-gate}.  A geometric search starting from
\(\tau_{\max}\) therefore terminates and accepts
\[
  \tau_k\geq
  \underline\tau_E
  :=\min\{\tau_{\max},\delta\tau_{\mathrm{safe},E}\}>0.
\]

For the accepted sequence, set
\(D_k=\mathcal L_\beta(u^k;\lambda^k)
-\mathcal L_\beta(u^{k+1};\lambda^k)\).  With
\(\bar u^k=\argmin_u\mathcal L_\beta(u;\lambda^k)\),
\(r^k=Eu^k-b\),
\(\bar r^k=E\bar u^k-b\), and
\(d_k=d_\beta(\lambda^k)\), direct expansion gives
\[
  V^k-V^{k-1}
  =-D_k+\tau_k\|r^k\|^2-2(d_k-d_{k-1}).
\]
Concavity of \(d_\beta\) and
\(\lambda^{k-1}-\lambda^k=\tau_kr^k\) imply
\[
  d_k-d_{k-1}
  \geq\tau_k\langle\bar r^k,r^k\rangle.
\]
Completing the square yields
\begin{equation}
\label{eq:adaptive-descent}
  V^k-V^{k-1}
  \leq
  \tau_k\|E(u^k-\bar u^k)\|^2-D_k
  -\tau_k\|\nabla d_\beta(\lambda^k)\|^2.
\end{equation}
Equations \eqref{eq:adaptive-gate} and
\eqref{eq:adaptive-majorant} imply
\[
  V^k-V^{k-1}
  \leq-\chi D_k-\tau_k\|\nabla d_\beta(\lambda^k)\|^2.
\]
The proof of Proposition~\ref{prop:general-small-dual-step} explicitly
constructs step-sequence-independent constants \(C_p,C_d>0\) such that
\[
  \Delta_p^k\leq(C_p/\gamma)D_k,\qquad
  \Delta_d^k\leq C_d\|\nabla d_\beta(\lambda^k)\|^2.
\]
Using \(\tau_k\geq\underline\tau_E\) gives
\begin{equation}
\label{eq:adaptive-contraction}
  V^k\leq(1+\eta)^{-1}V^{k-1},\qquad
  \eta=\min\left\{\frac{\chi\gamma}{C_p},
                  \frac{\underline\tau_E}{C_d}\right\}>0.
\end{equation}
This closes the proof for the general smooth strongly convex class.

\paragraph{Quadratic weighted refinement.}
If \(F,G\) are the strongly convex quadratics stated in the main text, set
\[
  K=\diag(Q_1,Q_2,0)+\beta E^\top E,\qquad
  p=\operatorname{prox}_h(u-\nabla s_\lambda(u))=u-v.
\]
Proximal optimality and monotonicity of \(\partial h\) give
\[
  e^\top Ke+\|v\|^2\leq e^\top(I+K)v.
\]
Cauchy--Schwarz in the \(K\)-metric, together with
\(K\succeq\beta E^\top E\), yields
\begin{equation}
\label{eq:adaptive-k-majorant}
\begin{aligned}
  e^\top(I+K)v
  &\leq
  \sqrt{e^\top Ke}\,
  \sqrt{v^\top(K^{-1}+2I+K)v},\\
  \|Ee\|^2
  &\leq\frac1\beta e^\top Ke
  \leq\frac1\beta v^\top(K^{-1}+2I+K)v
  =B_K(v).
\end{aligned}
\end{equation}
This step uses the quadratic identity
\(\nabla s_\lambda(u)-\nabla s_\lambda(\bar u)=Ke\).

The majorant can be evaluated without forming \(K^{-1}\).  For
\(v=(v_x,v_y,v_z)\),
\begin{align}
  v^\top K^{-1}v
  ={}&(v_x-A^\top v_z)^\top Q_1^{-1}(v_x-A^\top v_z)
  \notag\\
  &+(v_y-B^\top v_z)^\top Q_2^{-1}(v_y-B^\top v_z)
  +\beta^{-1}\|v_z\|^2,                                      \label{eq:adaptive-inverse}\\
  v^\top Kv
  ={}&v_x^\top Q_1v_x+v_y^\top Q_2v_y
  +\beta\|Av_x+Bv_y+v_z\|^2.                                 \label{eq:adaptive-hessian}
\end{align}
Thus two pre-factorized solves with \(Q_1\) and \(Q_2\) suffice.  The two
Bregman terms in \eqref{eq:adaptive-block-decrease} reduce to
\(\tfrac12\|\Delta x\|_{Q_1}^2\) and
\(\tfrac12\|\Delta y\|_{Q_2}^2\).

Set
\[
  c_K=\frac1\beta
  \max_{\xi\in\sigma(K)}(\xi+\xi^{-1}+2).
\]
Since \(B_K(v)\leq c_K\|v\|^2\), every
\[
  0<\tau\leq
  \tau_{\mathrm{safe},K}
  :=\frac{(1-\chi)\gamma}{c_K\sigma^2}
\]
satisfies \eqref{eq:adaptive-k-gate}, and geometric backtracking accepts
\[
  \tau_k\geq
  \underline\tau_K
  :=\min\{\tau_{\max},\delta\tau_{\mathrm{safe},K}\}>0.
\]
Replacing \eqref{eq:adaptive-gate}--\eqref{eq:adaptive-majorant} by
\eqref{eq:adaptive-k-gate}--\eqref{eq:adaptive-k-majorant} in
\eqref{eq:adaptive-descent} gives the same contraction, with
\(\underline\tau_K\) in place of \(\underline\tau_E\).

The trial-point evaluation order is essential: the trial residual
\(\widehat v\) and decrease \(\widehat D\) must use the same trial dual
iterate.  If a candidate step is
rejected, both trial iterates must be discarded, and the next candidate must
be constructed from the same \((u^k,\lambda^{k-1})\).  The theorem assumes
exact solves of the block subproblems and exact evaluation of the acceptance
test; the strict margin needed for a finite-precision acceptance decision
requires separate analysis.
}

\subsection{Proof of Theorem~\ref{thm:no-universal-dual-step}}
\label{app:no-universal-dual-step-proof}

\begin{proof}
It suffices to take \(\beta=1\).  For a prescribed
\(0<\vartheta\leq80/119\), set
\[
  r=\sqrt{\frac{7\vartheta}{80-7\vartheta}}\in(0,1/4],
  \qquad
  u=(1,r,0)^\top,\qquad
  v=(1,3r/5,4r/5)^\top,
\]
and define
\[
  P_A=\frac{uu^\top}{1+r^2},\qquad
  P_B=\frac{vv^\top}{1+r^2}.
\]
For \(\alpha\in[0,\sqrt2-1]\), let
\[
\begin{aligned}
  s&=\frac{2\alpha}{1+\alpha^2},\qquad
  t=\frac{1-\alpha^2}{1+\alpha^2},\\
  A_\alpha&=tP_A+s(I-P_A),&
  Q_{1,\alpha}&=s^2P_A+t^2(I-P_A),\\
  B_\alpha&=tP_B+s(I-P_B),&
  Q_{2,\alpha}&=s^2P_B+t^2(I-P_B).
\end{aligned}
\]
Then
\(Q_{1,\alpha}+A_\alpha^\top A_\alpha
 =Q_{2,\alpha}+B_\alpha^\top B_\alpha=I\).
For \(\alpha>0\), both Hessians are positive definite and
\(A_\alpha,B_\alpha\) are
nonsingular.  Take
\[
  x^\star=y^\star=0,\qquad
  z^\star=(1,0,0)^\top,\qquad
  \lambda^\star=(0,-1,-1)^\top,
\]
and set
\[
  b=z^\star,\qquad
  c_1=A_\alpha^\top\lambda^\star,\qquad
  c_2=B_\alpha^\top\lambda^\star.
\]
This is the unique strictly complementary KKT point.

Fix \(D=\diag(1,0,0)\).  The exact Schur recursion in
Appendix~\ref{app:universal-step-obstruction} proves
\[
  \rho(U_D(r,0))>1,
  \qquad
  \vartheta=\frac{80r^2}{7(1+r^2)}.
\]
At \(\alpha=\sqrt2-1\),
\[
  A_\alpha=B_\alpha=I/\sqrt2,\qquad
  Q_{1,\alpha}=Q_{2,\alpha}=I/2,
\]
and the coordinates decouple.  The scalar characteristic polynomials for
a positive-slack and a zero-slack coordinate are
\[
  \frac14(2\zeta-1)^2(\zeta+\vartheta-1),
  \qquad
  \frac{\zeta}{4}
  \bigl[4\zeta^2+(3\vartheta-5)\zeta+1-\vartheta\bigr].
\]
Write
\(p(\zeta)=4\zeta^2+(3\vartheta-5)\zeta+1-\vartheta\).
The Jury inequalities for this quadratic factor are
\[
  4>|1-\vartheta|,\qquad
  p(1)=2\vartheta>0,\qquad
  p(-1)=10-4\vartheta>0,
\]
so this endpoint is strictly Schur stable throughout the stated interval.

By continuity, some \(\alpha_c\in(0,\sqrt2-1)\) satisfies
\(\rho(U_D(r,\alpha_c))=1\).  The corresponding problem remains strongly convex
and full row rank.  The branch cannot have eigenvalue \(+1\): a sufficiently
small perturbation along its eigenvector would remain in the strict affine
branch and create a second ADMM fixed point, contradicting uniqueness of
the KKT point.

The unit-circle eigenvalue is therefore \(-1\) or a nonreal conjugate pair.
Choose a genuine eigenmode: a real eigenvector for \(-1\), or the real
invariant subspace spanned by the real and imaginary parts of a complex
eigenvector.  For a nonzero \(e\) in that mode,
\(\{U_D(r,\alpha_c)^ke\}\) alternates or rotates.  It is therefore bounded but
nonconvergent.  Scaling \(e\) below the strict KKT projection margin keeps
every iterate in the same branch, so the branch orbit is a bounded
nonconvergent raw ADMM orbit.
Because the unit-circle eigenvalue is nonzero, the two inactive components
of the \(z\)-part of \(e\) vanish; taking \(e\) still smaller preserves the
positive active component and gives \(z^0\in\R_+^3\).
Since \(A_{\alpha_c}\) is nonsingular, set
\[
  x^0=A_{\alpha_c}^{-1}(b-B_{\alpha_c}y^0-z^0).
\]
Then \(A_{\alpha_c}x^0+B_{\alpha_c}y^0+z^0=b\), so the initial state is fully
feasible.  The stored \(x^0\) does not enter the right-hand side of the next
\(x\)-subproblem and therefore does not alter the \((y,z,\lambda)\) orbit
above.  Each subsequent \(x\)-update is an affine function of that bounded
reduced state, so the full \((x,y,z,\lambda)\) orbit is bounded as well.
\end{proof}

\subsection{Uniform-Step Obstruction}
\label{app:universal-step-obstruction}

For \(e=(\delta y,\delta z,\delta\lambda)\),
\(R_A=A_\alpha A_\alpha^\top\), and
\(R_B=B_\alpha B_\alpha^\top\), define
\[
\begin{aligned}
  Y_y&=B_\alpha^\top R_AB_\alpha,&
  Y_z&=-B_\alpha^\top(I-R_A),&
  Y_\lambda&=B_\alpha^\top(I-R_A),\\
  C_y&=(I-R_B)R_AB_\alpha,&
  C_z&=R_A-R_BR_A+R_B,&
  C_\lambda&=I-R_A-R_B+R_BR_A.
\end{aligned}
\]
On a fixed projection mask \(D\), the full error matrix is
\[
U_D(r,\alpha)=
\begin{pmatrix}
Y_y&Y_z&Y_\lambda\\
DC_y&DC_z&DC_\lambda\\
\vartheta(I-D)C_y&\vartheta(I-D)C_z&
(1-\vartheta)I+\vartheta(I-D)C_\lambda
\end{pmatrix}.
\]
For completeness, we recall the real Schur recursion used below.
Let
\[
  p_j(\zeta)
  =
  \sum_{\ell=0}^{d_j}
  a_{j,\ell}\zeta^{d_j-\ell},
  \qquad
  p_j^\sharp(\zeta)
  :=
  \zeta^{d_j}p_j(\zeta^{-1}),
\]
where \(a_{j,0}\) and \(a_{j,d_j}\) are respectively the leading and
constant coefficients of \(p_j\).  Define
\[
  \Delta_j
  :=
  a_{j,0}^2-a_{j,d_j}^2
\]
and, whenever \(\Delta_j>0\),
\[
  p_{j+1}(\zeta)
  :=
  \frac{
    a_{j,0}p_j(\zeta)
    -a_{j,d_j}p_j^\sharp(\zeta)
  }{\zeta}.
\]
A necessary Schur--Cohn condition is that
\(\Delta_j>0\) at every recursion stage.  Consequently, once the preceding
recursions are well defined, a negative \(\Delta_j\) certifies that the
original polynomial has a zero outside the unit disk.
At \(\alpha=0\), \(D=\diag(1,0,0)\), and
\(\vartheta=80r^2/[7(1+r^2)]\), remove the four zero roots and the root
\(1-\vartheta\) from the characteristic polynomial.  The first three exact
Schur deltas of the remaining quartic satisfy
\[
  \Delta_0>0,\qquad \Delta_1>0,\qquad \Delta_2<0
  \qquad(0<r\leq1/4),
\]
where, after positive factors are removed, the signs of \(\Delta_1\) and
\(\Delta_2\) are determined by
\[
\begin{aligned}
  -h_1(r),\qquad
  h_1(r)&=1205r^8+1977r^6-4643r^4-7385r^2-2450,\\
  h_2(r),\qquad
  h_2(r)&=14530r^{10}+57137r^8-131931r^6\\
  &\hspace{4em}-449323r^4-329735r^2-82950.
\end{aligned}
\]
For \(0<r\le1\),
\[
  h_1(r)\le-1461r^4-7385r^2-2450<0,
\]
and
\[
  h_2(r)\le
  -60264r^6-449323r^4-329735r^2-82950<0.
\]
Thus the signs are strict and a root lies outside the unit disk.  The exact
polynomial, factorization, and endpoint Jury checks are regenerated by
\begin{center}
\small\ttfamily
python python/verify\_universal\_step\_obstruction.py\\
--check
\end{center}
from the GitHub certificate repository root.  The underlying exact script is
\begin{center}
\small\ttfamily
python/analyze\_identity\_slack\_universal\_step\_obstruction.py.
\end{center}

\subsection{Period-66 Verification}

From the GitHub certificate repository
\begin{center}
\url{https://github.com/ConanXu-math/identity-slack-admm-cycle-certificate}
\end{center}
run
\begin{center}
\texttt{python python/verify\_certificate\_pair.py}.
\end{center}
One-shot acceptance of both frozen certificates is
\texttt{python python/verify\_all.py}.
The command invokes separately implemented four- and six-dimensional rational
replays and verifies strong convexity, the unique KKT point, every ADMM
update, \(66\)-step closure, all \(132\) strict projection inequalities, and
minimality of the period.

\subsection{Perturbation Robustness}

\begin{corollary}[Persistence under data perturbations]
\label{cor:robustness}
Let \(\mathbb S_{++}^2\) denote the cone of real \(2\times2\) symmetric
positive-definite matrices.  With \(A=B=I_2\), \(\beta=1\), and the update
order \(x\to y\to z\to\lambda\) fixed, there is an open neighborhood of
\((Q_1,Q_2,\bar b)\) in
\(\mathbb S_{++}^2\times\mathbb S_{++}^2\times\R^2\) such that each resulting
problem has a non-KKT periodic ADMM sequence of minimal period \(66\), with
the same projection-sign sequence and every projection inequality strict.
\end{corollary}

\begin{proof}
For the fixed projection-sign sequence, \(\mathcal P\) and \(a\) in
\eqref{eq:period-map} depend analytically on the problem data.  The period
system remains nonsingular near the displayed instance, so
\((I_4-\mathcal P)^{-1}a\) and its \(66\) reduced iterates vary continuously.
The signs in \eqref{eq:strict-margin} have a positive common margin and remain
strict under sufficiently small perturbations.  The primitive word is
unchanged.
\end{proof}

\subsection{Relaxed Branch Map}
\label{app:period66-branch-map}

Set \(\beta=1\),
\[
  M=(Q_1+I_2)^{-1},\qquad N=(Q_2+I_2)^{-1},\qquad I=I_2,
\]
and use the six-dimensional essential state
\(w=(y,z,\lambda)\in\R^6\).  Eliminating \(x^+\) from the primal
updates gives
\[
  C_y=(I-N)M,\qquad
  C_z=N+(I-N)M,\qquad
  C_\lambda=(I-N)(I-M),\qquad
  d=C_\lambda\bar b,
\]
and
\[
  q^+=C_yy+C_zz+C_\lambda\lambda+d.
\]
For a strict projection mask
\[
  D=\diag\bigl(\mathbf 1_{\{q_1^+>0\}},
                \mathbf 1_{\{q_2^+>0\}}\bigr),
  \qquad z^+=Dq^+,
\]
the identity
\(
x^++y^++z^+-\bar b=\lambda-q^++z^+
\)
yields
\[
  \lambda^+=(1-\tau)\lambda+\tau(I-D)q^+.
\]
Thus the matrices in \eqref{eq:relaxed-affine-map} are
\[
\begin{aligned}
  T_D^{(0)}&=
  \begin{pmatrix}
    NM&-N(I-M)&N(I-M)\\
    DC_y&DC_z&DC_\lambda\\
    0&0&I
  \end{pmatrix},\\[0.4em]
  T_D^{(1)}&=
  \begin{pmatrix}
    0&0&0\\
    0&0&0\\
    (I-D)C_y&(I-D)C_z&(I-D)C_\lambda-I
  \end{pmatrix},
\end{aligned}
\]
and
\[
  a_D(\tau)=
  \begin{pmatrix}
    N(I-M)\bar b\\ Dd\\ 0
  \end{pmatrix}
  +\tau
  \begin{pmatrix}
    0\\ 0\\ (I-D)d
  \end{pmatrix}.
\]
This also shows directly that \(\tau\) affects only the multiplier block
row.

\subsection{Proof of Theorem~\ref{thm:relaxation}}
\label{app:period66-dual-step-proof}

\begin{proof}
Set \(T_0:=T_{01}^{(0)}\), \(T_1:=T_{01}^{(1)}\), and
\(T(\tau)=T_{01}(\tau)=T_0+\tau T_1\).  Define
\(\Phi_H(\tau)=H-T(\tau)^\top HT(\tau)\).  If
\(\tau=\eta\tau_-+(1-\eta)\tau_+\), direct expansion gives
\begin{equation}
\label{eq:chord-identity}
  \Phi_H(\tau)-\eta\Phi_H(\tau_-)
  -(1-\eta)\Phi_H(\tau_+)
  =\eta(1-\eta)(\tau_+-\tau_-)^2T_1^\top HT_1
  \succeq0.
\end{equation}
At \(\tau=1/2\), let \(H\in\Q^{6\times6}\) solve
\[
  H-T_{01}(1/2)^\top HT_{01}(1/2)=I_6.
\]
Exact Sylvester tests give
\(H\succ0\), \(\Phi_H(49/100)\succ0\), and
\(\Phi_H(51/100)\succ0\).  Equation~\eqref{eq:chord-identity} therefore
gives uniform contraction throughout \([49/100,51/100]\).  Strict
complementarity places a sufficiently small closed \(H\)-ellipsoid around
the KKT point entirely inside the \(D_{01}\) projection region, proving the
first assertion.

For the initialization in Theorem~\ref{thm:main}, a \(232\)-step rational
componentwise enclosure is propagated uniformly on
\(|\tau-1/2|\leq10^{-10}\).  Every projection sign remains strict, and the
entire enclosure at step \(232\) lies in the preceding projection-safe
ellipsoid.  The first assertion then proves convergence.

For the third assertion, exact factorization of the \(D_{01}\)-branch
characteristic polynomial, followed by a real Schur recursion and a Sturm
root count, reduces the stability boundary to the unique root in \((0,1)\)
of \(p_{\rm stab}\) in \eqref{eq:stability-polynomial}; exact endpoint evaluation gives
the stated bracket.  If \(0<\tau<\tau_c\), Schur stability gives the unique
\(H_\tau\succ0\) satisfying
\[
  H_\tau-T_{01}(\tau)^\top H_\tau T_{01}(\tau)=I_6.
\]
In the induced norm there is a \(\kappa_\tau<1\) such that
\[
  \|T_{01}(\tau)e\|_{H_\tau}
  \leq\kappa_\tau\|e\|_{H_\tau}.
\]
Since \(q^\star=(-1,1)\) is in the interior of the strict branch, a
sufficiently small closed \(H_\tau\)-ellipsoid is invariant under the full
projected map. Therefore
\[
  \|w^{k+1}-w^\star\|_{H_\tau}
  \leq
  \kappa_\tau\|w^k-w^\star\|_{H_\tau},
\]
so \(w^\star\) is locally Q-linearly attracting.

Conversely, suppose that \(\tau_c\leq\tau<1\).  Since \(w^\star\) lies in
the interior of the strict \(D_{01}\) branch, the full projected map agrees
with its affine branch map on a neighborhood of \(w^\star\).  Local
Q-linear attraction on that neighborhood would imply an induced norm in
which the linear part \(T_{01}(\tau)\) has operator norm strictly smaller
than one, and hence
\(\rho(T_{01}(\tau))<1\).  This contradicts
\(\rho(T_{01}(\tau))\geq1\).  Thus the KKT point is not locally
Q-linearly attracting for \(\tau_c\leq\tau<1\).
\end{proof}

\subsection{Verification of Theorem~\ref{thm:relaxation}}
\label{app:period66-dual-step-certificate}

The dual-step certificate is verified by
\begin{center}
\texttt{python python/certify\_relaxed\_multiplier\_interval\_theory.py}.
\end{center}
It re-derives the essential-state projection-region matrices from the original updates,
checks the endpoint Sylvester minors and chord identity, propagates the exact
\(232\)-step finite-entry enclosure, and performs the Schur--Sturm boundary
test.  The boundary polynomial is
\begin{equation}
\label{eq:stability-polynomial}
\begin{aligned}
p_{\rm stab}(\tau)={}&
111794210406295556649228900462157733493\,\tau^3\\
&+23105776975281816108275814441284422085171521\,\tau^2\\
&-244157339715898821440243649673959463071543521\,\tau\\
&+208410060660460340386576638889814578828638507.
\end{aligned}
\end{equation}
An exact Sturm count gives one root of \(p_{\rm stab}\) in \((0,1)\), and rational
endpoint evaluation gives the bracket in Theorem~\ref{thm:relaxation}.

\section{Period-23 Certificate}
\label{app:k3-period23}

This appendix records the reduced branch maps, the return-map construction,
the proofs, and the invariant-ellipsoid data for the exact \(m=3\)
period-\(23\) certificate stated in Section~\ref{sec:k3-route}
(Proposition~\ref{prop:k3-period23}).  It is a fixed-QP initialization result
and does not assert robustness to perturbations of the problem data.

\subsection{Branch Map}
\label{app:k3-formulation}

Record the state immediately after the \(z\)- and multiplier updates.  With
\(\beta=1\), set
\[
  t=z+\lambda,
  \qquad
  v=(y,t)\in\R^6,
  \qquad
  S_t=[\,0_{3\times3}\ I_3\,],
\]
so that \(t=S_tv\).  The projection identity gives
\(z=[t]_+\) and \(\lambda=[t]_-\); thus \(v\) determines
\((z,\lambda)\) and the next ADMM step.

Each strict projection branch is determined by
\[
  \sigma(t):=(\operatorname{sgn}t_1,\operatorname{sgn}t_2,
  \operatorname{sgn}t_3)\in\{-,+\}^3.
\]
On a fixed branch the projection is linear, and eliminating the \(x\)- and
\(z\)-updates gives
\[
  \Phi_\sigma(v)=R_\sigma v+r_\sigma,
  \qquad
  R_\sigma\in\Q^{6\times6},\quad r_\sigma\in\Q^6.
\]

For quadratic objectives
\[
  F(x)=\tfrac12x^\top Q_x x+c_1^\top x,\qquad
  G(y)=\tfrac12y^\top Q_y y+c_2^\top y,
\]
set
\[
  K_x=(Q_x+A^\top A)^{-1},\quad
  K_y=(Q_y+B^\top B)^{-1},
  \qquad H_x=AK_xA^\top,\quad H_y=BK_yB^\top,
\]
\[
  D_\sigma=\diag\!\bigl(\mathbf 1_{\{\sigma_i=+\}}\bigr),\quad
  J_\sigma=2D_\sigma-I_3,\quad E_\sigma=I_3-D_\sigma,
\]
\[
  \xi=K_x(A^\top b-c_1),\qquad
  \eta=K_y(B^\top b-c_2-B^\top A\xi).
\]
Direct substitution in the \(x\)-, \(y\)-, and projection--multiplier
updates gives
\begin{equation}
\label{eq:k3-one-step-branch}
  R_\sigma=
  \begin{bmatrix}
    K_yB^\top H_xB
      &K_yB^\top(H_x-I_3)J_\sigma\\
    (I_3-H_y)H_xB
      &E_\sigma+\bigl((I_3-H_y)H_x+H_y\bigr)J_\sigma
  \end{bmatrix},
  \qquad
  r_\sigma=
  \begin{bmatrix}
    \eta\\ b-A\xi-B\eta
  \end{bmatrix}.
\end{equation}
If the QP data are rational, then every entry of \(R_\sigma\) and
\(r_\sigma\) is rational.

\subsection{Exact Data and Certificate}
\label{app:k3-construction}

Apply the reduced map above to the rational instance
\eqref{eq:k3-rational-data} stated in Section~\ref{sec:k3-route}.
The resulting strict projection-sign word is \(\mathcal W_{23}\) in
\eqref{eq:k3-sign-word}.

Take \(\omega=\mathcal W_{23}=(\sigma_0,\ldots,\sigma_{22})\).
Write the composition of its first \(j\) branch maps as
\[
  \begin{aligned}
  \Phi_\omega^{(j)}(v)&=L_jv+d_j,
  &L_0&=I_6,\quad d_0=0,\\
  L_{j+1}&=R_{\sigma_j}L_j,
  &d_{j+1}&=R_{\sigma_j}d_j+r_{\sigma_j}.
  \end{aligned}
\]
The full \(23\)-step return map is therefore
\begin{equation}
\label{eq:k3-return-map}
  \begin{aligned}
  v^{23(\ell+1)}
    &=\Phi_{\rm per}\bigl(v^{23\ell}\bigr),\\
  \Phi_{\rm per}(v)
    &:=\Phi_\omega^{(23)}(v)
      =M_{\rm per}v+c_{\rm per},\\
  M_{\rm per}
    &:=L_{23}=R_{\sigma_{22}}\cdots R_{\sigma_0},
  \qquad
    c_{\rm per}:=d_{23}.
  \end{aligned}
\end{equation}
Exact elimination shows that \(I_6-M_{\rm per}\) is
nonsingular, so define the phase-zero state by
\[
  \widehat v^0:=(I_6-M_{\rm per})^{-1}c_{\rm per}\in\Q^6,
  \qquad
  \widehat v^j:=L_j\widehat v^0+d_j
  \quad (j=0,\ldots,22).
\]
The complete exact rational initialization is recovered by
\[
  \widehat t^0=S_t\widehat v^0,\qquad
  \widehat z^0=[\widehat t^0]_+,\qquad
  \widehat\lambda^0=[\widehat t^0]_-,
\]
\[
  \widehat K_x:=
  (\widehat Q_x+\widehat A^\top\widehat A)^{-1},
  \qquad
  \widehat x^0
  =\widehat K_x\!\left[
    \widehat A^\top
    (\widehat b-\widehat B\widehat y^{22}
     -\widehat z^{22}+\widehat\lambda^{22})
    -\widehat c_1
  \right].
\]
These equations define the exact rational initialization; the machine
certificate stores its expanded numerators and denominators.

\begin{proof}[Proof of Proposition~\ref{prop:k3-period23}]
For \(\omega=\mathcal W_{23}\), the recursion above gives the exact rational
return map
\[
  (M_{\rm per},c_{\rm per})=(L_{23},d_{23}).
\]
Exact rational elimination solves the fixed-point system for
\(\widehat v^0\).  Exact replay reconstructs every phase state and verifies
the pattern, the strict margin,
\(\Phi_{\rm per}(\widehat v^0)=\widehat v^0\), and that the \(23\) phase
states are pairwise distinct.  Hence the sequence has minimal period \(23\).

Positive definiteness of \(\widehat Q_x\) and \(\widehat Q_y\) gives a unique
primal solution; feasibility determines \(z\), and nonsingularity of
\(\widehat A\) determines \(\lambda\).  Thus the KKT point is unique and is a
fixed point of the single-valued ADMM map.  It cannot belong to the
nonconstant period-\(23\) sequence, which is therefore non-KKT.

For nearby initializations, the matrix \(P\) in
Appendix~\ref{app:k3-certificate} satisfies
\[
  P-M_{\rm per}^\top P M_{\rm per}\succ0.
\]
It follows that the error decreases after each complete block of \(23\)
ADMM steps.  In particular, every eigenvalue of \(M_{\rm per}\) lies strictly
inside the unit disk.
The strict sign margin guarantees that every point in
\(\mathcal E_{\rm cert}\) continues to use the same \(23\) affine branches.
The return map sends this ellipsoid into itself, so induction gives the same
projection pattern and phasewise convergence for all subsequent returns.
The \(23\) limiting phase points are distinct; hence the full sequence does
not converge to a single point.
\end{proof}

\begin{proof}[Proof of Corollary~\ref{cor:k3-identity-block-equivalence}]
Set \(u=\widehat A x\) and \(w=\widehat B y\).  Nonsingularity of
\(\widehat A\) and \(\widehat B\) makes the change of variables bijective and
preserves the residual
\(\widehat A x+\widehat B y+z-\widehat b=u+w+z-\widehat b\).
The transformed objectives are
\[
  \widetilde Q_x=\widehat A^{-\top}\widehat Q_x\widehat A^{-1},\quad
  \widetilde Q_y=\widehat B^{-\top}\widehat Q_y\widehat B^{-1},\quad
  \widetilde c_1=\widehat A^{-\top}\widehat c_1,\quad
  \widetilde c_2=\widehat B^{-\top}\widehat c_2,
\]
which remain rational with \(\widetilde Q_x,\widetilde Q_y\succ0\).
The \(x\)- and \(y\)-subproblems are invertible reparameterizations of the
\(u\)- and \(w\)-subproblems, while the \(z\)- and multiplier updates are
unchanged.  Starting from \(u^0=\widehat A x^0\) and \(w^0=\widehat B y^0\),
the two direct ADMM trajectories correspond step by step.
\end{proof}

\subsection{Attracting Ellipsoid}
\label{app:k3-certificate}

Use the partial-composition matrices \(L_j\) and the selector \(S_t\) defined
in Appendix~\ref{app:k3-formulation}.  Set
\[
P=\frac12
\begin{bmatrix}
 2& 0& 0& 0&-1& 1\\
 0& 2& 1& 0&-2& 2\\
 0& 1& 6& 2&-7& 7\\
 0& 0& 2& 3&-3& 3\\
-1&-2&-7&-3&16&-13\\
 1& 2& 7& 3&-13&14
\end{bmatrix}.
\]
Its leading principal minors are
\[
  1,\quad 1,\quad \frac{11}{4},\quad \frac{25}{8},\quad
  \frac{333}{32},\quad \frac{441}{32},
\]
so \(P\succ0\).  Exact Sylvester tests also give
\[
  P-M_{\rm per}^\top P M_{\rm per}\succ0.
\]
For the \(i\)th standard basis vector \(\mathbf e_i\in\R^3\), define
\[
  a_{ji}:=\mathbf e_i^\top S_tL_j\in\Q^{1\times6}.
\]
Thus \(a_{ji}e\) is the perturbation of the \(i\)th projection input at phase
\(j\) caused by the initial-state error \(e=v^0-\widehat v^0\).  Define the
corresponding sign-preserving radius by
\[
  \bar r^2:=
  \min_{\substack{0\leq j<23,\ 1\leq i\leq3\\a_{ji}\neq0}}
  \frac{(\widehat t_i^j)^2}{a_{ji}P^{-1}a_{ji}^\top}.
\]
The exact certificate gives
\(\bar r^2>29/100000>1/4000\); the minimum occurs at phase \(14\) and
coordinate \(t_3\).  The \(P\)-norm Cauchy--Schwarz inequality therefore
preserves all \(69\) projection signs for
\(e^\top Pe<1/4000\).  Since
\(P-M_{\rm per}^\top PM_{\rm per}\succ0\), the return map leaves
\(\mathcal E_{\rm cert}\) invariant.  Induction gives the same branch word
and phasewise convergence to the non-KKT sequence.

The ellipsoid in \eqref{eq:k3-certified-ellipsoid} is sufficient but not
claimed maximal.  On the slice \(\Delta y=0\), \(\Delta t_3=0\),
\[
  \frac32(\Delta t_1)^2
  -3\Delta t_1\Delta t_2
  +8(\Delta t_2)^2<\frac1{4000},
  \qquad
  \Delta t_2=0\ \Longrightarrow\
  |\Delta t_1|<\frac1{\sqrt{6000}}.
\]
Every point in this interval is therefore a certified nonconvergent reduced
initialization.

\subsection{Reproduction}
\label{app:k3-reproduction}

From the same GitHub certificate repository root, running
\begin{center}
\texttt{python python/verify\_period23\_certificate.py}
\end{center}
verifies positive definiteness, nonsingularity, the unique KKT point,
\(23\)-step closure, \(23\) distinct phase states, strict projection
consistency with margin \(>1/250\), separation from the KKT point, and the
exact Lyapunov inequality for the return matrix.  The rational input, verifier,
and generated certificate are provided in that repository.

\section{AI-Assisted Research Records}
\label{app:discovery-provenance}
\label{sec:protocol}
\label{sec:protocol-detail}

This appendix records route dossiers, workspace manifests, prompt classes,
human interventions, promotion criteria, and the claim-to-artifact map.  The
accompanying GitHub repository
\begin{center}
\small\url{https://github.com/ConanXu-math/identity-slack-admm-cycle-certificate}
\end{center}
stores certificates, verifiers, and retained research artifacts.  Route
dossiers live under
\path{provenance/routes/\{codex-period66,kimi-period23\}/}.
The discovery-to-verification workflow is summarized in 
Figure~\ref{fig:research-loop}
(Section~\ref{sec:ai-comparison}).

Computer-assisted claims are supported by exact certificates and independently
executable verifiers.  Analytic claims are established by the proofs in the
paper, with symbolic scripts used to check the stated algebraic and spectral
predicates where applicable.  Language-model transcripts carry no independent
evidentiary weight.

\subsection{Prompt Records}
\label{app:provenance-prompts}
\label{app:provenance-init}
\label{sec:protocol-init}

Archived prompts are classified into three provenance classes:
\begin{itemize}
\item \emph{verbatim historical} --- retained exactly as issued;
\item \emph{reconstructed} --- rebuilt from retained transcripts or notes;
\item \emph{retrospective distilled} --- written after the fact to summarize a
  stage.
\end{itemize}
No distilled prompt is asserted to be a historical input.  Route dossiers and
retained prompt artifacts live under \path{provenance/routes/} in the GitHub
repository; a consolidated prompt ledger may be archived alongside them.

Representative entries (full texts in the GitHub route dossiers):
P-K-01 (\texttt{verbatim\_historical}), Kimi opening goal;
P-C-01 (\texttt{reconstructed}), Codex brief;
P-C-02 (\texttt{retrospective\_distilled}), \(\tau\)-only relaxation with
explicit quantifier scopes.

\subsection{Human Interventions}
\label{app:provenance-interventions}
\label{sec:protocol-intervention}

We distinguish \emph{human-specified}, \emph{AI-proposed}, and
\emph{jointly refined} steps.  Human direction selected the mathematical
question, configured the workspace and tools, set claim scopes, and audited
literature; model outputs nominated representations, candidates, and
certificate designs.  Table~\ref{tab:intervention} records the staged
attribution for the Codex route.  Attributions are based on the retained
transcripts; where a representation emerged during the interaction and cannot
be attributed to a single side, it is labeled jointly refined.  Typical Codex
gates include rejecting fixed-branch instability alone as a counterexample,
requiring exact rational closure, and separating KKT-local from class-uniform
claims in the \(\tau\)-extension.  The Kimi dossier records the blind opening
and subsequent continue-search instructions under the frozen workspace rules.
Route dossiers under \path{provenance/routes/} retain the supporting
artifacts.

\begin{table}[H]
\centering
\caption{Recorded stages of the Codex route, with the promotion condition
that gated each stage.}
\label{tab:intervention}
\small
\begin{tabular}{
  >{\raggedright\arraybackslash}p{0.185\textwidth}
  >{\raggedright\arraybackslash}p{0.24\textwidth}
  >{\raggedright\arraybackslash}p{0.25\textwidth}
  >{\raggedright\arraybackslash}p{0.22\textwidth}}
\toprule
\textbf{Stage} & \textbf{Human input} & \textbf{AI output} &
\textbf{Promotion condition}\\
\midrule
Problem setup & Identity-slack question; evidence requirements &
Proof and counterexample route proposals &
Correct literature boundary\\
State-space reduction & Request for a certifiable representation &
Signed reduced state \(s=(y,q)\), \(q=z+\lambda\) (jointly refined) &
Equivalence to raw ADMM\\
Candidate screening & Requirement that fixed-branch instability alone is
insufficient & Fixed-branch spectra; near-rotation heuristic &
Projection itinerary must be realizable\\
Period search & Requirement of a strict exact counterexample &
Rational data \eqref{eq:short-data}--\eqref{eq:b-data} and word
\((00)^2(01)^{64}\) & Exact rational closure and strict signs
(\(G_1\)--\(G_7\))\\
Relaxation extension & Frozen QP and initialization; vary only \(\tau\);
quantifier gates & Parameterized branch maps; candidate step regimes;
certificate routes & Quantifier audit and exact certificates\\
\bottomrule
\end{tabular}
\end{table}

\subsection{Promotion Gates}
\label{app:provenance-gates}
\label{sec:protocol-gates}
\label{sec:protocol-artifacts}

A candidate becomes a mathematical claim only after passing
\[
\begin{aligned}
  G_1&:\ \text{exact problem data,} &
  G_2&:\ \text{KKT and subproblem validity,}\\
  G_3&:\ \text{exact period closure,} &
  G_4&:\ \text{strict branch admissibility,}\\
  G_5&:\ \text{raw ADMM replay,} &
  G_6&:\ \text{minimal-period verification,}\\
  G_7&:\ \text{quantifier and literature audit.} &&
\end{aligned}
\]
Gates \(G_1\)--\(G_6\) are mathematical or computational checks, carried out
in exact arithmetic where applicable; \(G_7\) is a human audit of quantifiers,
novelty, and literature scope.  Floating-point orbits and model-generated
derivations nominate candidates only.  For stability claims, \(G_3\)--\(G_6\)
are replaced by the relevant Lyapunov, Schur, Sturm, or interval predicates.
In particular, KKT-local convergence, convergence from a specified
initialization, fixed-problem guarantees, and class-uniform guarantees are
treated as distinct promotion scopes rather than stylistic qualifications.

\subsection{Claim-to-Artifact Map}
\label{app:provenance-claims}

Verification entry points in the GitHub repository (first exact replayable
certificate as endpoint):
\begin{center}
\scriptsize
\begin{tabular}{ll}
\toprule
\textbf{Claim} & \textbf{Verification command}\\
\midrule
Theorem~\ref{thm:main} &
  \texttt{python python/verify\_certificate\_pair.py}\\
Theorem~\ref{thm:relaxation} &
  \texttt{python python/certify\_relaxed\_multiplier\_interval\_theory.py}\\
Theorem~\ref{thm:no-universal-dual-step} &
  \texttt{python python/verify\_universal\_step\_obstruction.py --check}\\
Proposition~\ref{prop:k3-period23} &
  \texttt{python python/verify\_period23\_certificate.py}\\
One-shot acceptance (66+23) &
  \texttt{python python/verify\_all.py}\\
\bottomrule
\end{tabular}
\end{center}
The computer-assisted claims listed in the table rest on these exact replays,
independently of language-model transcripts.
\end{document}